\documentclass{amsart}
\usepackage{amsmath}
\usepackage{amssymb}
\usepackage{array}
\usepackage{epsfig}
\usepackage{color}
\usepackage{comment}
\usepackage{setspace}
\usepackage{textcomp}
\usepackage{stmaryrd}
\usepackage{multirow}
\usepackage{diagbox}
\usepackage{booktabs}
\usepackage{subcaption}
\usepackage{url}
\numberwithin{equation}{section}

\newtheorem{theorem}{Theorem}

\newtheorem{remark}{Remark}

\DeclareMathOperator{\Div}{div}

\DeclareMathOperator{\Grad}{\nabla}
\DeclareMathOperator{\ssum}{\textstyle \sum}

\newcommand{\pmat}[1]{\begin{pmatrix} #1 \end{pmatrix}}

\newcommand{\beq} {\begin{equation}}
\newcommand{\eeq} {\end{equation}}
\newcommand{\bdm} {\begin{displaymath}}
\newcommand{\edm} {\end{displaymath}}

\newcommand{\bit}{\begin{itemize}}
\newcommand{\eit}{\end{itemize}}
\newcommand{\bde}{\begin{description}}
\newcommand{\ede}{\end{description}}

\newcommand{\ben}{\begin{enumerate}}
\newcommand{\een}{\end{enumerate}}
\newcommand{\algn}[1]{\begin{align} #1 \end{align}}
\newcommand{\algns}[1]{\begin{align*} #1 \end{align*}}
\newcommand{\mltln}[1]{\begin{multline} #1 \end{multline}}
\newcommand{\mltlns}[1]{\begin{multline*} #1 \end{multline*}}
\newcommand{\gat}[1]{\begin{gather} #1 \end{gather}}

\newcommand{\barr}{\begin{array}}
\newcommand{\earr}{\end{array}}
\newcommand{\subeqns}[2]{\begin{subequations} \label{#1} \algn{#2} \end{subequations}}
\newcommand{\mc}[1]{\mathcal{#1}}
\newcommand{\LRp}[1]{\left( #1 \right)}

\newcommand{\LRs}[1]{\left[ #1 \right]}
\newcommand{\LRa}[1]{\left< #1 \right>}

\newcommand{\e}{\epsilon}

\newcommand{\I}{\Bbb{I}}

\renewcommand{\div}{\operatorname{div}}

\newcommand{\bs}{\boldsymbol}

\newcommand{\lap}{\Delta}
\newcommand{\pd}{\partial}

\newcommand{\norw}[2]{\left\Vert#1\right\Vert_{#2}}

\newcommand{\n}{\bs{n}}

\newcommand{\inner}[2]{\langle #1, #2 \rangle}
\newcommand{\foralls}{\forall \,}

\newcommand{\alb}{\bs{\alpha}}
\newcommand{\ahi}{a_{h,i}}

\newcommand{\ds}{\, \mathrm{d} s}

\newcommand{\Dub}{D_{\ub}}
\newcommand{\Dpt}{D_{\pt}}
\newcommand{\Dpi}{D_{\pii}}
\newcommand{\Dpb}{D_{\pb}}

\newcommand{\dub}{\dot{\bs{u}}}

\newcommand{\dpt}{\dot{p}_{t}}

\newcommand{\dpb}{\dot{\pb}}

\newcommand{\dpii}{\dot{p}_{i}}

\newcommand{\esssup}{\operatorname{ess \, sup}}

\newcommand{\ept}{e_{\pt}}

\newcommand{\eptin}[1]{e_{\pt}^{I,n}}
\newcommand{\epthn}[1]{e_{\pt}^{h,n}}

\newcommand{\epi}{e_{p_i}}

\newcommand{\epb}{e_{\pb}}

\newcommand{\eu}{e_{\bs{u}}}

\newcommand{\fb}{\bs{f}}

\newcommand{\laminv}{\lambda^{-1}}

\newcommand{\half}{\frac{1}{2}}

\newcommand{\Ib}{\bs{I}}

\newcommand{\pb}{\bs{p}}
\newcommand{\pbh}{\bs{p}_h}
\newcommand{\pt}{p_{t}}
\newcommand{\pth}{p_{t,h}}
\newcommand{\pii}{p_{i}}
\newcommand{\piih}{p_{i,h}}

\newcommand{\qb}{\bs{q}}
\newcommand{\Qb}{\bs{Q}}
\newcommand{\Qbh}{\bs{Q}_h}
\newcommand{\qt}{q_{t}}
\newcommand{\qii}{q_{i}}
\newcommand{\Qt}{Q_{t}}
\newcommand{\Qth}{Q_{t,h}}
\newcommand{\Qi}{Q_{i}}
\newcommand{\Qih}{Q_{i,h}}

\newcommand{\R}{{\mathbb R}}

\newcommand{\triang}{\mathcal{T}}

\newcommand{\ub}{\bs{u}}
\newcommand{\ubh}{\bs{u}_h}

\newcommand{\Vbh}{\bs{V}_h}
\newcommand{\Vb}{\bs{V}}
\newcommand{\vb}{\bs{v}}

\newcommand{\wb}{\bs{w}}

\newcommand{\X}{{\mathbb X}}
\newcommand{\xiji}{\xi_{j \leftarrow i} }

\newcommand{\strain}{\varepsilon}

\title[Partitioned methods for MPET]{Unconditionally stable second order convergent partitioned methods for multiple-network poroelasticity} 

\author{Jeonghun J. Lee}
\address{Department of Mathematics, Baylor University, Sid Richardson Science Building, One Bear Place  97328, Waco, Texas 76798, USA }
\email{Jeonghun\_Lee@baylor.edu}

\begin{document}

\maketitle

\begin{abstract}
  In this paper, we consider partitioned numerical methods for quasi-static multiple-network poroelasticity (MPET) equations, 
  generalizations of the Biot model in poroelasticity for multiple pore networks. 
  Two partitioned numerical methods are presented for the equations 
  which split time discretization into solving two subequations, a Lam\'e equation and a system of heat equations, alternatively.
  In contrast to the iterative coupling methods which require multiple iterations at each time step, our numerical methods solve these smaller equations only once at each time step. 
  We prove their unconditional stability and high order convergence in time with a novel error analysis. 
  A number of numerical results are presented to illustrate good performances of these partitioned methods.
\end{abstract}

%% REQUIRED
%\begin{keywords}
%  multiple-network poroelasticity, partitioned numerical methods, nearly incompressible materials
%\end{keywords}
%
%% REQUIRED
%\begin{AMS}
%65M12, 65M15, 65M60%, 92C10
%\end{AMS}

%-------------------------------------------------------------------------------
%-------------------------------------------------------------------------------
\section{Introduction}
In this paper, we consider partitioned numerical methods for a 
family of quasi-static multiple-network
poroelasticity (MPET\footnote{The abbreviation MPET is from the
  term multiple-network poroelastic theory in literature 
  e.g.~\cite{TullyVentikos2011}. Here, we instead refer to the
  multiple-network poroelasticity equations but keep the abbreviation
  for the sake of convenience.}) equations reading as follows: for a
given number of networks $N \in \mathbb{N}$, find the displacement $\ub$
and the network pressures $p_{i}$ for $i = 1, \dots, N$ such that
\begin{subequations}
  \label{eq:mpet}
  \begin{align}
    \label{eq:mpet:1}
    - \Div \mc{C} \strain(\ub) + \sum_{i=1}^N \alpha_{i} \Grad p_{i} &= \fb, \\
    \label{eq:mpet:2}
    s_{i} \dot{p}_{i} + \alpha_{i} \Div \dub - \Div \LRp{ K_{i} \Grad p_{i} } + \xi_{i} (\pb) &= g_{i}, \qquad 1 \le i \le N,
  \end{align}
\end{subequations}
where $\pb$ is the $N$-tuple $(p_1, \cdots, p_N)$. Here $\ub = \ub(x, t)$, $p_i = p_i(x, t)$, $1\le i \le N$ are 
time-dependent functions for $x \in \Omega \subset \R^{d}$ ($d = 2, 3$) and for $t \in [0, T]$.
The operators and parameters are as
follows: $\mc{C}$ is the elastic stiffness tensor, each network $i$ is
associated with a Biot-Willis coefficient $\alpha_{i} \in (0, 1]$,
storage coefficient $s_{i} \geq 0$, and hydraulic conductivity tensor $K_{i} = \kappa_{i}/\mu_{i} > 0$ (where $\kappa_{i}$ and $\mu_{i}$ represent the network permeability and the network fluid
viscosity, respectively). In~\eqref{eq:mpet:1}, $\Grad$, $\strain$, and $\Div$
denote the (column-wise) gradient, the symmetric (row-wise) gradient, and the row-wise divergence, respectively. 
In~\eqref{eq:mpet:2}, $\Grad$ and
$\Div$ are the standard gradient and divergence operators, and the
superposed dot denotes the time derivative. Further, $\fb$ represents
a body force and $g_{i}$ represents sources in network $i$ for $i =
1, \dots, N$, while $\xi_{i}$ represents transfer terms from network
$i$ to other networks.
From a physical point of view, \eqref{eq:mpet} represents the balance of linear momentum and the mass conservation
in a porous, linearly elastic medium permeated by $N$
segregated viscous fluid networks. 

In this paper, we consider the case of an isotropic stiffness tensor
for which
\begin{equation}
  \label{eq:def:isotropic}
  \mc{C} \strain(\ub) = 2 \mu \strain(\ub) + \lambda \Div \ub \Bbb{I}
\end{equation}
where $\mu, \lambda$ are the standard non-negative Lam\'e parameters
and $\Bbb{I}$ denotes the identity tensor. Moreover, we will consider the
case where the transfer terms $\xi_i$, quantifying the transfer out of
network $i$ into the other fluid networks, are proportional to
pressure differences between the networks. More precisely, we assume
that $\xi_i$ takes the form:
\begin{align}
  \label{eq:def:transfer}
  \xi_i(\pb ) = \sum_{i=1}^N \xiji (p_j - p_i), 
\end{align}
where $\xi_{j \leftarrow i}$ are non-negative transfer coefficients
for $1 \le i, j \le N$. We will also assume that these transfer
coefficients are symmetric in the sense that $\xi_{j \leftarrow i} =
\xi_{i \leftarrow j}$, and note that $\xi_{j \leftarrow j}$ is
arbitrary.

The MPET equations have an abundance of both geophysical and biological applications. 
For example, the case with $N = 2$ is the Biot-Barenblatt model which models dual porosity 
property of poroelastic media. 
%Our interest in the multiple-network poroelasticity equations
%primarily stems from the use of these equations in modelling
%interactions between biological fluids and tissue in physiological
%settings. As one example, 
In biomechanics, Tully and Ventikos~\cite{TullyVentikos2011}
considered~\eqref{eq:mpet} with four different networks ($N = 4$) as a macroscopic 
model for the dynamics of fluid flows in brain tissue.
%In this context, the original system of \cite{TullyVentikos2011} represents a macroscopic model of interaction between the different fluid networks in the brain. 
The fluid networks represent the arteries, the
arterioles/capillaries, the veins and the interstitial fluid-filled
extracellular space, and each network may have a different permeability
$\kappa_i$ and different transfer coefficients $\xi_{j \leftarrow i}$
for $1 \le j \le N$, $j \not = i$.
 
%and we note that Showalter and Momken~\cite{ShowalterMomken2002} present an existence analysis for this model, 

While the Biot model has been throughly studied, see e.g.~\cite{%Showalter2000,
AguilarEtAl2008,BaerlandLeeEtAl2017, Lee2017,MuradEtAl1996,OyarzuaRuizBaier2016,PhillipsWheeler2007,    Yi2017}, 
only a few numerical methods for the MPET equations are available. To the best of our knowledge, 
the first numerical method was proposed in \cite{LeeEtAl2018}, and another method is proposed more recently
in \cite{Hong-et-al-2018}. 
Robust preconditioners of numerical methods for MPET equations can be constructed by
extending the block preconditioners of the Biot model (see e.g., \cite{Hong-Kraus-2018,LeeEtAl2017}). 
%-----
%There are enormous research works on numerical methods for Biot's equations
%such as \cite{PhillipsWheeler2007, berger2015stabilized, bause2017space, hu2017nonconforming,rodrigo2017new,murad} to name a few. 
%In recent studies \cite{Yi2017, Yi2014, KorsaweStarke2005,Lee2016}, new formulations of
%Biot's model using three or four unknowns have been studied 
%in order to overcome the volumetric locking for nearly incompressible
%poroelastic materials. This is an important property to use the poroelastic models 
%for biological tissue modelling because most soft biological tissues are 
%modeled as (nearly) incompressible materials for short time-scales and physiological
%scales of pressures. Therefore, it may be crucial for accurate modeling of the
%interaction of the different network pressures in~\eqref{eq:mpet} to
%allow for an elastic material that is almost incompressible and/or
%with (nearly) vanishing storage coefficients, i.e.,~for $1 \ll \lambda
%< +\infty$ and $0 < s_i \ll 1$ in~\eqref{eq:mpet}. Standard two-field
%mixed finite element discretizations of the Biot model, approximating
%the displacement and the fluid pressure only using stable mixed finite elements for 
%the Stokes equations, are well-known to perform poorly in the incompressible
%limit, see e.g.~\cite{Lee2017} and references therein. 
Nevertheless, the problem sizes of the MPET equations are intrinsically large, 
so partitioned numerical methods for time discretization are valuable 
approaches to reduce computational costs.

There are only a few of previous studies on partitioned numerical methods 
for poroelasticity equations. %In [Bukac and Layton et al] 
The conditional stability of various partitioned methods for a dynamic poroelasticity model
was studied in \cite{bukac2015analysis}.
In \cite{ChaabaneRiviere2018} a partitioned numerical method was developed 
for quasi-static poroelasticity equations using the discontinuous Galerkin method with a stabilization technique.
However, the error analysis of the method gives $1/2$ convergence rate of time discretization errors
which is not regarded as the optimal order of time discretization errors. 
%In addition, an additional stabilization term is introduced in the method but determination of
%optimal parameter values in the stabilization term depends on physical parameter values. 

The objective of this paper is to develop and analyze new partitioned numerical schemes for the MPET equations. 
We propose two partitioned numerical schemes which are unconditionally stable without any stabilization terms 
and have second or higher order convergence in time. A novel error analysis will be presented to show the stability and convergence
of the partitioned schemes, and numerical results will be given to illustrate their performances.
We use the formulation proposed in \cite{LeeEtAl2018} with the total pressure, so the implicit constants in the
error estimates are uniformly bounded for arbitrarily large $\lambda >0$ and for small storage coefficients $s_i$'s. 
In contrast to monolithic numerical methods, our partitioned numerical 
schemes solve two subproblems sequentially, one is a linear elasticity equation 
and the other is a system of parabolic equations,
so required computational resources at each solve can be significantly reduced.
We point out that our partitioned schemes are intrinsically different
from the iterative coupling methods, which are equivalent to 
monolithic methods with block triangular preconditioners which use exact LU solvers
in diagonal blocks of the preconditioners.

This paper is organized as follows. Section~\ref{sec:notation}
presents notation and general preliminaries. In Section~\ref{sec:mpet:qs}, we introduce 
a variational formulation with the total pressure %~\eqref{eq:mpet:vf} 
for the quasi-static MPET equations~\eqref{eq:mpet}. 
In Section~\ref{sec:error-analysis} we present two partitioned numerical methods
and prove their convergence with the a priori error estimates.
We present numerical results in Section~\ref{sec:numerics:convergence} to illustrate 
these theoretical results. 
%corroborated by synthetic numerical convergence experiments in
Conclusions and future research directions are highlighted in
Section~\ref{sec:conclusion}.

%-------------------------------------------------------------------------------
%-------------------------------------------------------------------------------
\section{Notation and preliminaries}
\label{sec:notation}

In this paper we use $X \lesssim Y$ to denote an inequality
$X \leq C Y$ with a generic constant $C > 0$ which is independent of
mesh sizes. If needed, we will write $C$
explicitly in inequalities but it can vary across expressions.

\subsection{Sobolev spaces}
Let $\Omega$ be a bounded polyhedral domain in $\R^d$ ($d= 2,3$) with boundary $\partial \Omega$. We let $L^2(\Omega)$ be the set
of square-integrable real-valued functions on $\Omega$. The inner
product of $L^2(\Omega)$ and the induced norm are denoted by
$\inner{\cdot}{ \cdot}$ and $\| \cdot \|$, respectively. For a
finite-dimensional inner product space $\X$, typically $\X = \R^d$,
let $L^2(\Omega; \X)$ be the space of $\X$-valued functions such that
each component is in $L^2(\Omega)$. The inner product of $L^2(\Omega;
\X)$ is naturally defined by the inner product of $\X$ and
$L^2(\Omega)$, so we use the same notation $\inner{\cdot}{\cdot}$ and
$\| \cdot \|_0$ to denote the inner product and norm on $L^2(\Omega;
\X)$. 
%For a non-negative real-valued function on $\Omega$ (or
%symmetric positive semi-definite tensor-valued function on $\Omega$)
%$w$, we also introduce the short-hand notations
%\begin{equation}
%  \label{eq:def:weighted:L2:norm}
%  \inner{u}{v}_{w} = \inner{w u}{v},
%  \quad \| u \|_{w}^2 = \inner{u}{u}_{w} ,
%\end{equation}
%noting that the latter is a norm only when $w$ is strictly positive
%a.e.~on $\Omega$ (or is positive definite a.e.~on $\Omega$).

For a non-negative integer $m$, $H^m(\Omega)$ denotes the standard
Sobolev spaces of real-valued functions based on the $L^2$-norm, and
$H^m(\Omega; \X)$ is defined similarly based on $L^2(\Omega; \X)$. To
avoid confusion with the weighted $L^2$-norms we use $\| \cdot \|_{m}$ to
denote the $H^m$-norm (both for $H^m(\Omega)$ and $H^m(\Omega;
\X)$). For $m \geq 1$, we use $H^m_{0, \Gamma}(\Omega)$ to denote the
subspace of $H^m(\Omega)$ with vanishing trace on $\Gamma \subset
\partial \Omega$, and $H^m_{0, \Gamma}(\Omega; \X)$ is defined
similarly~\cite{Evans1998}. For $\Gamma = \partial \Omega$, we write
$H^m_0(\Omega)$ and analogously $H^m_{0}(\Omega; \X)$.

\subsection{Spaces involving time}

For $T >0$ and a reflexive Banach
space $\mathcal{X}$, let $C^0 ([0, T] ; \mathcal{X})$ denote the set
of functions $f : [0, T] \rightarrow \mathcal{X}$ that are continuous
in $t \in [0, T]$. For an integer $m \geq 1$, we define
\begin{equation*}
  C^m ([0, T]; \mathcal{X}) = \{ f \, | \, \partial^{i}f/\partial t^{i} \in C^0([0, T];\mathcal{X}), \, 0 \leq i \leq m \},
\end{equation*}
where $\partial^i f/\partial t^i$ is the $i$-th time derivative in the
sense of the Fr\'echet derivative in $\mathcal{X}$ (see
e.g.~\cite{Yosida1980}).

For a function $f : [0, T] \rightarrow \mathcal{X}$, the Bochner norm is defined as 
\begin{equation*}
  \| f \|_{L^r((0, T); \mathcal{X})} =
  \begin{cases}
    \left( \int_0^T \| f(s) \|_\mathcal{X}^r \ds \right)^{1/r}, \quad 1 \leq r < \infty, \\
    \esssup_{t \in (0, T)} \| f (t) \|_\mathcal{X}, \quad r = \infty.
  \end{cases}
\end{equation*}
We define $W^{k,r}((0, T); \mathcal{X})$
for a non-negative integer $k$ and $1 \leq r \leq \infty$ as the
closure of $C^k ([0, T]; \mathcal{X})$ with the norm $\| f
\|_{W^{k,r}((0, T);\mathcal{X})} = \sum_{i=0}^k \| \partial^i f /
\partial t^i \|_{L^r((0, T); \mathcal{X})}$.

\subsection{Finite element spaces}

Let $\triang_h$ be a shape-regular triangulation of $\Omega$. For any integer $k \geq 1$, we let
$\mathcal{P}_k(\triang_h)$ denote the space of continuous piecewise
polynomials of order $k$ associated to $\triang_h$, and
$\mathcal{P}_k(\triang_h; \R^d)$ as the space of $d$-tuples with
components in $\mathcal{P}_k(\triang_h)$. We will omit 
$\triang_h$ when it is clear in context. 

\subsection{Parameter values}

We will make the following assumptions on the material parameter
values. First, we assume that the Biot-Willis coefficients $\alpha_i
\in (0, 1]$, $i = 1, \dots, N$, and the storage coefficients $s_i > 0$
are constant in time for $i = 1, \dots, N$. In the analysis, we will
pay particular attention to robustness of estimates with respect to
arbitrarily large $\lambda$ because large $\lambda$ correspond to 
nearly incompressible materials which are common in biomechanical
modelling.
  
We will assume that the hydraulic conductivities $K_i$ are constant in
time, but possibly spatially-varying and that these satisfy standard
ellipticity constraints: i.e.~there exist positive constants $K_i^{-}$
and $K_i^{+}$ such that
\begin{equation*}
  K_i^{-} \leq K_i(x) \leq K_i^{+} \quad \text{a.e.} \;  x \in \Omega.
\end{equation*}
We assume that the transfer coefficients $\xi_{j \leftarrow i}$ are
constant in time and non-negative, i.e.,~$\xi_{j \leftarrow i}(x) \geq
0$ for $1 \leq i, j \leq N$, $x \in \Omega$.

\subsection{Boundary conditions}

We will consider~\eqref{eq:mpet} augmented by the following standard
boundary conditions. First, we assume that the boundary decomposes into
two parts: $\partial \Omega = \Gamma_d \cup \Gamma_t$ with $|\Gamma_d
\cap \Gamma_t| = 0$ and $|\Gamma_d|, |\Gamma_t| > 0$ where
$|\Gamma|$ is the $(d-1)$-dimensional Lebesgue measure of $\Gamma \subset \partial \Omega$.  We use $\n$ to denote
the outward unit normal vector field on $\partial \Omega$.  Relative
to this partition, we consider the homogeneous boundary conditions
\begin{subequations}
  \label{eq:mpet:bcs}
  \begin{align}
    \ub &= 0 \quad \text{ on } \Gamma_d , \\
    \left( \mc{C} \strain(\ub) - \sum_{i=1}^N \alpha_i p_i \I \right) \n &= 0 \quad \text{ on } \Gamma_t , \\
    p_i &= 0 \quad \text{ on } \partial \Omega \quad \text{ for } i = 1, \dots, N . 
  \end{align}
\end{subequations}
The subsequent formulation and analysis can easily be extended to
cover inhomogeneous and other types of boundary conditions with 
suitable modifications.

%\subsection{Key inequalities}
%
%For the space $\Vb = H^1_{0, \Gamma_d}(\Omega)$, Korn's
%inequality~\cite[p.~288]{Braess2001} holds; i.e., there exists a
%constant $C > 0$ depending only on $\Omega$ and $\Gamma_d$ such that
%\begin{align}
%  \label{eq:mpet:korn}
%  \| \vb \|_0 \le C \|\strain(\vb) \|_0 \quad \foralls \vb \in \Vb .  
%\end{align}
%
%Furthermore, for the combination of spaces $\Vb$ and $Q_t =
%L^2(\Omega)$, the following (continuous Stokes) inf-sup condition
%holds: there exists a constant $C > 0$ depending only on $\Omega$ and
%$\Gamma_d$ such that
%\begin{align} 
%\label{eq:mpet:infsup} 
%\sup_{\vb \in \Vb} \frac{\inner{\Div \vb}{q}}{\| \vb \|_0} \ge C \| q \|_0 \quad \foralls q \in L^2(\Omega) .   
%\end{align} 
%Our discretization schemes will also satisfy corresponding discrete
%versions of Korn's inequality and the inf-sup condition with constants
%independent of the discretization.

\subsection{Initial conditions}

The MPET equations~\eqref{eq:mpet} need appropriate initial conditions. 
In particular, in agreement with the
assumption that $c_i > 0$ for $i = 1, \dots, N$, we assume that
initial conditions are given for all $p_i$:
\begin{equation}
  \label{eq:mpet:ics}
    p_i(x, 0) = p_i^0(x), \quad x \in \Omega, \quad i = 1, \dots, N .
\end{equation}
Given such $p_i^0$, we note that we may compute $\ub(x, 0) = \ub^0(x)$
from~\eqref{eq:mpet:1}, which in particular yields a $\Div \ub(x, 0) =
\Div \ub^0(x)$ for $x \in \Omega$. In the following, we will assume that
any initial conditions given are compatible in the sense described
here.

%-------------------------------------------------------------------------------
%-------------------------------------------------------------------------------
\section{The formulation with total pressure}
\label{sec:mpet:qs}

In this section, we review the formulation for the
quasi-static multiple-network poroelasticity equations with the total pressure.
%Inspired by~\cite{OyarzuaRuizBaier2016, LeeEtAl2017}, we introduce an additional variable, namely the \emph{total pressure}. 
In order to be consistent with the pressure definitions in solid mechanics total pressure definition in this paper and the one in \cite{LeeEtAl2018} have different sign.
In the subsequent subsections, we present the augmented governing equations and
introduce a corresponding variational formulation.

%-------------------------------------------------------------------------------
\subsection{Governing equations introducing the total pressure}

Let $\ub$ and $p_i$ for $i = 1, \dots, N$ be solutions
of~\eqref{eq:mpet} with boundary conditions given
by~\eqref{eq:mpet:bcs}, initial conditions given
by~\eqref{eq:mpet:ics} and recall the isotropic stiffness tensor
assumption, cf.~\eqref{eq:def:isotropic}. 

For simplicity, we denote $\alb = (\alpha_1, \dots,
\alpha_N)$ and $\pb = (p_1, \dots, p_N)$, and we write
%\begin{equation*}
$\alb \cdot \pb = \ssum_{i=1}^N \alpha_i p_i$. 
%\end{equation*}
Introducing the total pressure $\pt$ defined as
%\begin{equation}
%  \label{eq:def:p_0}
$-\pt = \lambda \Div \ub - \alb \cdot \pb$,
%\end{equation}
we have 
\begin{equation}
  \label{eq:divu:as:p0}
  \Div \ub = \laminv \LRp{ -\pt + \alb \cdot \pb } .
\end{equation}

Inserting~\eqref{eq:divu:as:p0} and its time-derivative
into~\eqref{eq:mpet:2}, we obtain an augmented system of quasi-static
multiple-network poroelasticity equations: for $t \in (0, T]$, find
  the displacement vector field $\ub$ and the pressure scalar fields
$\pt$ and $\pb$ such that
\begin{subequations}
  \label{eq:mpet:tp}
  \begin{align}
    \label{eq:mpet:tp1} - \Div \left ( 2 \mu \strain (\ub) - \pt \Bbb{I} \right ) &= \fb, \\
    \label{eq:mpet:tp2} - \Div \ub - \laminv \pt + \laminv \alb \cdot \pb &= 0, \\ 
    \label{eq:mpet:tp3} s_i \dot{p}_i + \alpha_i \laminv (-\dpt + \alb \cdot \dot{\pb}) - \Div (K_i \nabla p_i) + \xi_i(\pb) &= g_i, \quad i = 1, \dots, N .
  \end{align}
\end{subequations} 
We note that $\pt(x, 0)$ can be computed from~\eqref{eq:mpet:ics}
and the definition of $\pt$.

%-------------------------------------------------------------------------------
\subsection{Variational formulation}

For $\Gamma_d \subset \pd \Omega$ satisfying $0 < |\Gamma_d | < |\pd \Omega|$, using the notations introduced in Section~\ref{sec:notation}, let
\begin{equation}
\Vb = H^1_{0, \Gamma_d} (\Omega; \R^d), \quad \Qt = L^2(\Omega), \quad Q_i =
H^1_0(\Omega) \quad i = 1, \dots, N .
\end{equation}
If $\Gamma_d = \pd \Omega$, then $\Qt = L_0^2(\Omega)$, the space of all mean-value zero functions in $L^2(\Omega)$. 
If $\Gamma_d = \emptyset$, then $\Vb$ must be the subspace of $H^1(\Omega; \R^d)$ orthogonal to the space of rigid motions on $\Omega$. 
For simplicity of presentation we will assume that $0 < |\Gamma_d | < |\pd \Omega|$ in the rest of this paper.
We also use $\Qb = Q_1 \times \dots \times Q_N$ for simplicity. Here we define the norms of $\norw{\cdot}{\Vb}$, $\norw{\cdot}{\Qt}$ as
\algns{
\norw{\vb}{\Vb}^2 =  \LRa{ 2\mu \strain(\vb), \strain(\vb) } , \qquad \norw{\qt}{\Qt}^2 = \LRa{ (2\mu)^{-1} \qt, \qt } .
}
Throughout the paper we assume that $\mu \lesssim \lambda$, therefore $\| \qt \|_{\laminv} \le C \| \qt \|_{\Qt}$ holds with a constant $C$ which is uniformly bounded above for arbitrarily large $\lambda$.

Multiplying~\eqref{eq:mpet:tp} by test functions and integrating by
parts with boundary conditions given by~\eqref{eq:mpet:bcs} and
initial conditions given by~\eqref{eq:mpet:ics} yield the following
variational formulation: given a compatible initial data $(\ub^0, \pt^0, \pb^0)$ satisfying \eqref{eq:mpet:tp1}, \eqref{eq:mpet:tp2}, and given
$\fb$ and $g_i$'s for $i = 1, \dots, N$, find $\ub \in C^0([0, T]; \Vb)$, $\pt \in
C^0([0,T]; \Qt) \cap C^1((0, T], \Qt)$, and $\pb \in
C^0([0,T]; \Qb) \cap C^1((0, T], \Qb)$ such that
\begin{subequations}
  \label{eq:mpet:vf}
  \begin{alignat}{2}
  \label{eq:mpet:vf:1}
    \inner{2 \mu \strain(\ub)}{\strain(\vb)} - \inner{\pt }{\Div \vb} &= \inner{\fb}{\vb} &&\quad \foralls v \in \Vb, \\
    \label{eq:mpet:vf:2}
    - \inner{\Div \ub}{\qt} - \LRa{\laminv \pt, \qt } + \inner{\laminv \alb \cdot \pb}{\qt} &= 0 &&\quad \foralls \qt \in Q_t , \\
    \inner{s_i \dot{p}_i + \alpha_i \laminv ( -\dpt + \alb \cdot \dot{\pb}  + \xi_i(\pb)}{\qii} + \LRa{K_i \Grad \pii, \Grad \qii } &= \inner{g_i}{\qii}
    &&\quad \foralls \qii \in Q_i
  \end{alignat}
\end{subequations}
for $i = 1, \dots, N$ and such that $\ub(\cdot, 0) = \ub^0(\cdot)$, $\pt(\cdot, 0) = \pt^0(\cdot)$, and
$p_i(\cdot, 0) = p_i^0(\cdot)$ for $i = 1, \dots, N$.

\section{Partitioned numerical methods with error analysis} \label{sec:error-analysis}

In this section, we present two partitioned time discretization algorithms and their a priori error analyses. 
In Subsection \ref{subsec:prelim}, we introduce notations and definitions for the algorithms and the associated error analyses.
In addition, we prove estimates of initial errors which are commonly necessary for error analyses of the two algorithms. 
In Subsection \ref{subsec:etd}, the first one, called ``elasticity-then-diffusion'' algorithm, will be defined 
and the error analysis will be presented with the full details. 
In Subsection \ref{subsec:dte}, we will present another algorithm, so called ``diffusion-then-elasticity'' algorithm. 
Since the error analysis of the diffusion-then-elasticity method is similar to the one in Subsection \ref{subsec:etd}, we will state main intermediate estimates in the analysis and details will be omitted.

In our error analyses the first method (elasticity-then-diffusion) methods has the second order convergence in time whereas the second method (diffusion-then-elasticity) has the third order convergence in time. In spite of this lower order convergence rate, the first method can be advantageous when local mass conservation is significant because then numerical solutions at the same time step are used in the mass conservation equations, so a locally mass conservative numerical flux can be easily recovered by post-processing. This is not the case in the second method because the numerical solutions of $\pb$ and $\pt$ are at different time steps as we will see in Subsection \ref{subsec:dte}.

\subsection{Preliminaries} \label{subsec:prelim}
In our numerical algorithms, we take time steps $t_n = n \lap t$ for $n = 0, 1, \ldots$ and given time step size $\lap t >0$. 
For a function $\sigma$ in $C^0([0, T]; \mc{X})$, we use $\sigma^n$ to denote $\sigma(t_n)$.
$\sigma_h$ is a discrete solution of $\sigma$ if $\sigma$ is a variable in equations and $\sigma_h^n$ is the $n$-th time step solution of $\sigma_h$. 
For error terms we define $e_\sigma^n = \sigma^n - \sigma_h^n$. For all quantities with superscript indices, 
we will use the convention 
\algns{
\sigma^{n+\half} = \frac{\sigma^n + \sigma^{n+1}} 2 .
}
For the finite element discretization of $\ub$ and $\pt$ we use a pair of finite element spaces 
$\Vbh \subset \Vb$, $\Qth \subset \Qt$
which satisfy the inf-sup condition 
\begin{align} 
\label{eq:mpet:infsup} 
\sup_{\vb \in \Vbh} \frac{\inner{\Div \vb}{q}}{\norw{\vb}{\Vb}} \ge C \norw{q}{\Qt} \quad \foralls q \in \Qth    
\end{align} 
with $C>0$ independent of the mesh sizes. For discretization of $p_i$'s for $i=1, \ldots, N$, we use a 
finite element method for the primal form of the Poisson equation yielding a symmetric bilinear form
for $\LRa{K_i \nabla p_i, \nabla q_i}$. Such methods include the continuous Galerkin (CG) methods, the discontinuous or enriched Galerkin 
methods (DG or EG) with symmetric bilinear forms (see e.g., \cite{Arnold-82, Arnold-Brezzi-et-al-02, Lee-Lee-Wheeler-16}), and the finite element space is denoted by $\Qih$ for $i=1, \ldots, N$. 
In order to keep this generality, we use $\ahi( \cdot , \cdot )$ to denote the discrete bilinear form corresponding to $\LRa{K_i \nabla \pii, \nabla \qii}$. The corresponding discrete norm $\| \cdot \|_{\ahi}$ is defined by
\algns{
\| \qii \|_{\ahi}^2 = \ahi(\qii, \qii) .
}
The convergence orders of numerical solutions depend on the approximation properties of $\Vbh$, $\Qth$, $\Qih$ for $i=1, \ldots, N$. 
For simplicity of presentation we assume that the approximation properties of $\Vbh$ and $\Qth$ satisfy
\algns{
\inf_{\vb \in \Vbh} \| \ub - \vb \|_{\Vb} \lesssim h^{k_{\ub}} \| \ub \|_{k_{\ub}+1} , \qquad \inf_{\qt \in \Qt} \| \pt - \qt \|_{\Qt} \lesssim h^{k_{\ub}} \| \pt \|_{k_{\ub}}
}
for a positive integer $k_{\ub}$ if $\ub$ and $\pt$ are sufficiently regular. This assumption holds for the family of Taylor--Hood elements \cite{Boffi94,Boffi97,BrezziFalk91,TaylorHood73} ($\mc{P}_{r+1}(\R^d) - \mc{P}_r$, $r \ge 1$) with $k_{\ub} = r+1$ and for the MINI element \cite{ArnoldBrezziFortin84} with $k_{\ub} =1$. Similarly, assuming $Q_{1,h} = \cdots = Q_{N,h}$ for simplicity, we assume that 
\algns{
\inf_{\qii \in \Qih} \| \pii - \qii \|_{\ahi} \lesssim h^{k_p} \| \pii \|_{k_{p}+1} \qquad 1 \le i \le N
}
for $k_p \ge 1$ and for sufficiently regular $\pb$. This holds for CG methods with $\Qih = \mc{P}_{k_p}$, or DG/EG methods with piecewise polynomials of degree $k_p$. 
For convenience we define two additional bilinear forms 
\algn{
\label{eq:S-form} & S(\pb, \qb) := \sum_{i=1}^N  \LRa{ s_i \pii, \qii } + \LRa{ \laminv \alb \cdot \pb, \alb \cdot \qb} , \\
\label{eq:A-form} & A(\pb, \qb) := \sum_{i=1}^N  \LRp{ \LRa{ \xi_i (\pb), \qii} + \ahi(\pii, \qii) } ,
}
with the corresponding norms $\| \cdot \|_S$ and $\| \cdot \|_A$. By a discrete Poincar\'e inequality 
$\| \qb \|_S \lesssim \| \qb \|_A$ holds for all $\qb \in \Qbh$ with an implicit constant which is uniform for small $s_i$'s
and arbitrarily large $\lambda$.

The continuous solutions (with a regularity assumption $\pii \in H^s(\Omega)$ for $s > 3/2$ if DG or EG is used for the discretization of $\Qi$) satisfy the variational equations:
\subeqns{eq:method1-conti}{
\label{eq:method1-conti1} &\LRa{ 2\mu \strain(\ub^{n}), \strain(\vb) } - \LRa{\pt^{n}, \div \vb} = \LRa{\fb^{n}, \vb}, \\
\label{eq:method1-conti2} &-\LRa{\Div \ub^{n}, \qt} - \LRa{\laminv \pt^{n}, \qt} = -\LRa{\laminv \alb \cdot \pb^{n}, \qt}, \\
\label{eq:method1-conti3} &-\LRa{ s_i \dpii^{n+\half} , \qii} - \LRa{\alpha_i \laminv \alb \cdot \dpb^{n+\half} + \xi_i \LRp{ \pb^{n+\half} } , \qii } \\
&\notag \quad  - \ahi \LRp{ \pii^{n+\half} , \qii } = -\LRa{ g^{n+\half} , \qii} - \LRa{\alpha_i \laminv \dpt^{n+\half} , \qii}, \qquad  1 \le i \le N,
%\label{eq:method1-conti3} &\LRa{ s_i \dpii^{n+1} , \qii} + \LRa{\alpha_i \laminv \alb \cdot \dpb^{n+1} + \xi_i \LRp{ \pb^{n+1} } , \qii } + \ahi \LRp{\pii^{n+1} , \qii } \\
%&\notag \qquad = \LRa{ g^{n+1}, \qii} - \LRa{\alpha_i \laminv \dpt^{n+1} , \qii}, \qquad  1 \le i \le N,
}
for $\vb \in \Vbh$, $\qt \in \Qth$, $\qii \in \Qih$.
For the error analysis we split error terms into two parts using appropriate interpolation operators. 
The interpolation operators for the variables $\ub$, $\pt$, $\pii$ for $i=1, \ldots, N$ will be denoted by $\Pi_h^{\Vb}$, $\Pi_h^{\Qt}$, 
$\Pi_h^{\Qi}$ for $i=1, \ldots, N$. 
Specifically, we define $(\Pi_h^{\Vb} \ub, \Pi_h^{\Qt} \pt)$ as the solution of the Lam\'{e} equation
\algns{ 
\LRa{ 2 \mu \strain (\Pi_h^{\Vb} \ub), \strain(\vb) } - \LRa{\Pi_h^{\Qt} \pt, \div \vb} &= \LRa{ 2 \mu \strain (\ub), \strain(\vb) } - \LRa{\pt, \div \vb} , & & \vb \in \Vbh, \\
- \LRa{\Div \Pi_h^{\Vb} \ub, \qt } - \LRa{ \laminv \Pi_h^{\Qt} \pt, \qt} &= - \LRa{\Div \ub, \qt } - \LRa{ \laminv \pt, \qt} , & & \qt \in \Qth .
}
For simplicity, we use $\Pi_h^{\Qb} \pb$ to denote the $N$-tuple $(\Pi_h^{Q_1} p_1, \ldots, \Pi_h^{Q_N} p_N )$. 
Then $\Pi_h^{\Qb} \pb$ is defined as the solution of the elliptic system
\algns{ 
\sum_{i=1}^N \LRs{ \ahi(\Pi_h^{\Qi} \pii, \qii) + \LRa{\xi_i(\Pi_h^{\Qb} \pb), \qii } } = \sum_{i=1}^N \LRs{ \ahi \LRp{ \pii, \qii} + \LRa{\xi_i(\pb), \qii } }, \qquad \qii \in \Qih
}
for $1 \le i \le N$. 
Well-posedness of this problem is not difficult to show from the property of $\xi_i$ 
\algns{
\sum_{1 \le i \le N} \LRa{\xi_i(\pb), \qii} = \sum_{1 \le i, j \le N} \half \LRa{\xiji (\pii - p_j), \qii - q_j} ,
}
which, in particular, gives 
\algns{
\sum_{1 \le i \le N} \LRa{\xi_i(\pb), \pii} = \sum_{1 \le i, j \le N} \half \LRa{\xiji (\pii - p_j), \pii - p_j} \ge 0 .
}
By a standard error analysis argument, we can show that these interpolation operators have optimal approximation properties
in $H^1$ norm for $\Vbh$, in $L^2$ norm for $\Qth$, and in (discrete) $H^1$ norm for $\Qih$, i.e., 
\algn{
\label{eq:intp-estm1} \| \ub - \Pi_h^{\Vb} \ub \|_{\Vb} + \| \pt - \Pi_h^{\Qt} \pt \|_{\Qt} &\lesssim h^{k_{\ub}} \LRp{ \| \ub \|_{k_{\ub}+1} + \| \pt \|_{k_{\ub}} }, \\
\label{eq:intp-estm2} \| \pb - \Pi_h^{\Qb} \pb \|_A + \| \pb - \Pi_h^{\Qb} \pb \|_S &\lesssim h^{k_p} \sum_{i=1}^N \| \pii \|_{k_p +1} 
}
for sufficiently regular $\ub$, $\pt$, $\pb$. 

For the error analysis we use the notation 
\algn{ \label{eq:err-decomp}
e_{\ub}^n &= \ub^n - \ub_h^n = \ub^n - \Pi_h^{\Vb} \ub^n + \Pi_h^{\Vb} \ub^n - \ub_h^n =: e_{\ub}^{I,n} + e_{\ub}^{h,n}, \\
e_{\pt}^n &= \pt^n - \pth^n = \pt^n - \Pi_h^{\Qt} \pt^n + \Pi_h^{\Qt} \pt^n - \pth^n =: e_{\pt}^{I,n} + e_{\pt}^{h,n}, \\
e_{\pii}^n &= \pii^n - \piih^n = \pt^n - \Pi_h^{\Qi} \pii^n + \Pi_h^{\Qi} \pii^n - \piih^n =: e_{\pii}^{I,n} + e_{\pii}^{h,n}, \quad 1 \le i \le N.
}
%where $\Pi_h^{\Vb} \ub^n$ is an interpolation of $\ub^n$ which will be defined below. We define $e_{\pt}^n$, $e_{\pt}^{I,n}$, $e_{\pt}^{h,n}$ and $e_{\pii}^n$, $e_{\pii}^{I,n}$, $e_{\pii}^{h,n}$ for $i=1, \ldots, N$ similarly with interpolation operators $\Pi_h^{\Qt}$, $\Pi_h^{\Qi}$ for $i = 1, \ldots, N$, which are defined below as well.
We also define 
\algns{
D_{\sigma}^{n} := e_{\sigma}^{h, n+1} - e_{\sigma}^{h, n} %\qquad 
%S_{\sigma}^{n} = \half \LRp{e_{\sigma}^{h, n+1} + e_{\sigma}^{h, n} } .
}
for variable $\sigma$. 

Throughout the paper we assume that $(\ubh^{0}, \pth^{0}, \pbh^{0})$, our discrete initial data, satisfies the following:
%is an optimal approximation of $(\ub^0, \pt^0, \pb^0)$ in $\| \cdot \|_{\Vb}$, $\| \cdot \|_{\Qt}$, $\| \cdot \|_{A}$ norms, i.e., 
\gat{ 
\label{eq:initial-assumption1} \pb_h^0 = \Pi_h^{\Qb} \pb^0, \qquad \| \ub^0 - \ubh^{0} \|_{\Vb} + \| \pt^0 - \pth^{0} \|_{\Qt} \lesssim h^{k_{\ub}} \LRp{\| \ub^0 \|_{k_{\ub}+1} + \| \pt^0 \|_{k_{\ub}} }. %\quad \| \pb^0 - \pbh^{0} \|_A \lesssim h^{k_{\pb}} 
}
As a consequence of \eqref{eq:initial-assumption1} and the inequality $\| \cdot \|_S \lesssim \| \cdot \|_A$, we have 
\algn{ \label{eq:init-h-approx}
\| \eu^{h,0 } \|_{\Vb} + \| \ept^{h,0} \|_{\Qt} \lesssim h^{k_{\ub}} , \quad \| \epb^{h,0} \|_S + \| \epb^{h,0} \|_A \lesssim h^{k_{\pb}} 
}
in which the implicit constants of these inequalities depend on the norms of $\ub^0$, $\pt^0$, $\pb^0$. 

To define our partitioned algorithms we need $(\ubh^1, \pth^1, \pbh^1)$ as well. To obtain this, we use the monolithic method combining the backward Euler and Crank--Nicolson methods, i.e.,  $(\ubh^1, \pth^1, \pbh^1)$ satisfies
\subeqns{eq:monolithic-disc}{ 
\label{eq:monolithic-disc1} & \LRa{ 2\mu \strain(\ubh^{1}), \strain(\vb) } - \LRa{\pth^{1}, \div \vb} = \LRa{\fb^{1}, \vb}, \\
\label{eq:monolithic-disc2} & -\LRa{\Div \ubh^{1}, \qt} - \LRa{\laminv \pth^{1}, \qt} + \LRa{\laminv \alb \cdot \pbh^{1}, \qt} = 0, \\
\label{eq:monolithic-disc3} -\LRa{ s_i \frac{\piih^{1} - \piih^0 }{\lap t} , \qii} &- \LRa{- \alpha_i \laminv \frac{\pth^{1} - \pth^{0}}{\lap t} + \alpha_i \laminv \alb \cdot \frac{\pbh^{1} - \pbh^0}{\lap t} + \xi_i \LRp{  \piih^{\half} } , \qii } \\
&\notag - \ahi \LRp{ \piih^{\half} , \qii } = - \LRa{ g_i^{\half} , \qii} , \quad 1 \le i \le N .
}
For the error estimates, we need estimates of some error terms in the beginning time steps. 
\begin{theorem} \label{thm:init-estm} For given compatible initial data $(\ub^0, \pt^0, \pb^0)$ and given $\fb$ and $\{ g_i \}_{i=1}^N$, suppose that $(\ub, \pt, \pb)$ is the solution of \eqref{eq:mpet:tp}. For numerical initial data $(\ubh^0, \pth^0, \pbh^0)$ satisfying \eqref{eq:initial-assumption1} suppose that $(\ubh^1, \pth^1, \pbh^1)$ is obtained by \eqref{eq:monolithic-disc}. Then we have 
\algn{
\label{eq:initD-estm} \| \Dub^0 \|_{\Vb}^2 + \| \Dpt^0 \|_{\Qt}^2 + \| \Dpb^{0} \|_{S}^2 + \lap t \| \epb^{h,1} \|_A^2 &\lesssim (\lap t)^6 + h^{2k_{\pb}} + h^{2 k_{\ub}}  , \\
\label{eq:init-estm} \| \eu^{h,1} \|_{\Vb} + \| \ept^{h,1} \|_{\Qt} + \| \epb^{h,1} \|_S &\lesssim (\lap t)^3 + h^{k_{\pb}} + h^{k_{\ub}} ,
}
and the implicit constants in these inequalities depend on the norms of the exact solutions and are uniformly bounded above for small $s_i$'s and arbitrarily large $\lambda$.
\end{theorem}
\begin{proof}
From $- \div \ub^0 - \laminv \pt^0 = - \laminv \alb \cdot \pb^0$ we see that 
\algns{
- \div e_{\ub}^0 - \laminv \ept^0 &= - \laminv \alb \cdot \pb^0 + \div \ubh^0 + \laminv \pth^0 = - \laminv \alb \cdot \epb^0 + R_h^0
}
where $R_h^0 := \div \ubh^0 + \laminv \pth^0 - \laminv \alb \cdot \pbh^0$.
Using this and the interpolation operators defined above we have
\algn{
\label{eq:n0-eq1} \LRa{ 2\mu \strain(\eu^{h, 0}), \strain(\vb) } - \LRa{\ept^{h, 0}, \div \vb} &=  \LRa{2 \mu \strain(\eu^{0}), \strain(\vb)} - \LRa{ \ept^{0}, \div \vb} , \\
\label{eq:n0-eq2} - \LRa{\Div \eu^{h,0}, \qt} - \LRa{\laminv \ept^{h,0}, \qt} &= - \LRa{\laminv \alb \cdot (\epb^{h,0} + \epb^{I,0}) - R_h^0, \qt} .
%\notag &\quad + \LRa{\div \eu^{0}, \qt} + \LRa{\laminv \ept^{0}, \qt}  ,
}
Moreover, the differences of \eqref{eq:monolithic-disc1}, \eqref{eq:monolithic-disc2} and \eqref{eq:method1-conti1}, \eqref{eq:method1-conti2} with $n=1$ with the interpolation operators $\Pi_h^{\Vb}$ and $\Pi_h^{\Qt}$, give 
\algn{
\label{eq:n1-eq1} \LRa{ 2\mu \strain(\eu^{h, 1}), \strain(\vb) } - \LRa{\ept^{h, 1}, \div \vb} &= 0, \\
\label{eq:n1-eq2} - \LRa{\Div \eu^{h,1}, \qt} - \LRa{\laminv \ept^{h,1}, \qt} &= - \LRa{\laminv \alb \cdot (\epb^{h,1} + \epb^{I,1}), \qt} .
}
The differences of these equations are
\algn{
\label{eq:tmp1} \LRa{ 2\mu \strain(\Dub^{0}), \strain(\vb) } -  \LRa{\Dpt^{0}, \div \vb} &= - \LRa{2 \mu \strain(\eu^{0}), \strain(\vb)} + \LRa{ \ept^{0}, \div \vb}, \\
\label{eq:tmp2} - \LRa{\Div \Dub^{0}, \qt} - \LRa{\laminv \Dpt^{0}, \qt} &= - \LRa{\laminv \alb \cdot (\Dpb^{0} + \epb^{I,1} - \epb^{I,0} ) + R_h^0 , \qt} .
%\notag &\quad - \LRa{\div \eu^{0}, \qt} - \LRa{\laminv \ept^{0}, \qt}  .
}
The difference of \eqref{eq:method1-conti3} and \eqref{eq:monolithic-disc3} with $\lap t$ multiple in consideration of the interpolation $\Pi_h^{\Qb}$ is 
\algn{
\label{eq:tmp3} 
& -\LRa{ s_i \Dpi^0 , \qii} -  \LRa{ \alpha_i \laminv (- \Dpt^0 + \alb \cdot \Dpb^0 ) + \lap t \xi_i \LRp{ \epb^{h,\half} } , \qii } \\ 
\notag &\qquad - \lap t \ahi \LRp{ \epi^{h,\half} , \qii } =  \lap t \LRa{s_i \Ib_{3,i}^0 + \alpha_i \laminv \alb \cdot \Ib_3^0 - \alpha_i \laminv I_4^0, \qii} 
}
in which $\Ib_3^n = (\Ib_{3,1}^n, \ldots, \Ib_{3,N}^n)$, $I_4^n$ are defined by
\algn{ \label{eq:I3-I4}
\Ib_{3,i}^n =  \dpii^{n+\half} - \frac{\Pi_h^{\Qi} \pii^{n+1} - \Pi_h^{\Qi} \pii^{n}}{\lap t} , \qquad 
I_{4}^n =  \dpt^{n+\half} - \frac{\Pi_h^{\Qt} \pt^{n+1} - \Pi_h^{\Qt} \pt^{n}}{\lap t} .
}
By the inf-sup condition \eqref{eq:mpet:infsup}, there exists $\wb^{0} \in \Vbh$ such that 
\algn{ \label{eq:inf-sup}
-\LRa{ \Div \wb^{0}, \Dpt^0} = \| \Dpt^0 \|_{\Qt}^2, \qquad \| \wb^{0} \|_{\Vb} \lesssim \| \Dpt^{0} \|_{\Qt} .
}
With a sufficiently small $\delta >0$ independent of the mesh sizes, we can get 
\algn{ \label{eq:stokes-stability}
&C_0 \| \Dub^{0} \|_{\Vb}^2 + C_0 \| \Dpt^0 \|_{\Qt}^2 \\ %+ \| \Dpt^0 \|_{\laminv}^2 \\
&\notag \le \LRa{ 2\mu \strain(\Dub^{0} ), \strain(\Dub^0 + \delta \wb^0) } - \LRa{\Dpt^{0} , \div (\Dub^0 + \delta \wb^0) } + \LRa{\Div \Dub^{0} , \Dpt^0} %+ \LRa{\laminv \Dpt^{0}, \Dpt^0} 
}
with $C_0 >0$ independent of the mesh sizes.

If we take $\vb = \Dub^0 + \delta \wb^{0}$, $\qt = - \Dpt^0$, $\qii = -\Dpi^0$ in \eqref{eq:tmp1}, \eqref{eq:tmp2}, 
\eqref{eq:tmp3}, and add them altogether, then we get  
\algn{ \label{eq:D0-estm}
&C_0 \| \Dub^0 \|_{\Vb}^2 + C_0 \| \Dpt^0 \|_{\Qt}^2 + \| \Dpt^0 - \alb \cdot \Dpb^0 \|_{\laminv}^2 + \sum_{i=1}^N \| \Dpi^0 \|_{s_i}^2 %+ \| \alb \cdot \Dpb^0 \|_{\laminv}^2 + \| \Dpb^0 \|_{S}^2 
+ \half \lap t \| \epb^{h,1} \|_A^2  \\
\notag &\le - \LRa{ 2 \mu \strain(\eu^{0}), \strain(\Dub^0 + \delta \wb^0) } + \LRa{ \ept^{0} , \div (\Dub^0 + \delta \wb^0) } + \LRa{\laminv \alb \cdot (\epb^{I,1} - \epb^{I,0}), \Dpt^0 } \\
\notag &\quad + \LRa{R_h^0 , \Dpt^0} + \half \lap t \| \epb^{h,0} \|_A^2  - \lap t \LRp{ S(\Ib_3^0 , \Dpb^0) + \LRa{ \laminv I_4^0 , \alb \cdot \Dpb^0 } }  
%&\quad + \LRa{\laminv \ept^{0}, \Dpt^0 } .
}
using \eqref{eq:stokes-stability}. We first remark that 
\algns{
|\LRa{R_h^0 , \Dpt^0}| \lesssim \| e_{\ub}^0 \|_{\Vb} \| \Dpt^0 \|_{\Qt} + \LRp{ \| \ept^0 \|_{\laminv} + \| \alb \cdot \epb^0  \|_{\laminv} } \| \Dpt^0 \|_{\laminv}
}
holds due to the equality $- \div \ub^0 - \laminv \pt^0 = - \laminv \alb \cdot \pb^0$. Moreover, 
\algn{
\notag (\lap t)^2 \| \Ib_{3,i}^n \|_0^2 &= \| \lap t \dpii^{n+\half} - (\pii^{n+1} - \pii^n) + (\epi^{I,n+1} - \epi^{I,n}) \|_0^2 \\
\label{eq:I3-estm} &\lesssim (\lap t)^6 \| \pii \|_{W^{3,\infty}(t_n, t_{n+1}; L^2)}^2 + (\lap t)^2 h^{2k_{p}} \| \pii \|_{W^{1,\infty}(t_n, t_{n+1}; H^{k_{p}+1})}^2, \\
\notag (\lap t)^2 \| I_4^n \|_0^2 &\lesssim \| \lap t \dpt^{n+\half} - (\pt^{n+1} - \pt^n) + (\ept^{I,n+1} - \ept^{I,n}) \|_0^2 \\
\label{eq:I4-estm}&\lesssim (\lap t)^6 \| \pt \|_{W^{3,\infty}(t_n, t_{n+1}; L^2)}^2 + (\lap t)^2 h^{2k_{\ub}} \| \pt \|_{W^{1,\infty}(t_n, t_{n+1}; H^{k_{\ub}})}^2 
}
hold. We also remark that 
\algns{
\| \Dpb^0 \|_S \lesssim  \| \Dub^0 \|_{\Vb}^2 + \| \Dpt^0 \|_{\Qt}^2 + \| \Dpt^0 - \alb \cdot \Dpb^0 \|_{\laminv}^2 + \sum_{i=1}^N \| \Dpi^0 \|_{s_i}^2
}
holds by the triangle inequality $\| \alb \cdot \Dpb^0 \|_{\laminv} \leq \| \Dpt^0 \|_{\laminv} + \| \Dpt^0 - \alb \cdot \Dpb^0 \|_{\laminv}$. 
Then applying Young's inequality to \eqref{eq:D0-estm}, we have 
\algns{
&\| \Dub^0 \|_{\Vb}^2 +  \| \Dpt^0 \|_{\Qt}^2 +  \| \Dpt^0 - \alb \cdot \Dpb^0 \|_{\laminv}^2 + \sum_{i=1}^N \| \Dpi^0 \|_{s_i}^2 + \lap t \| \epb^{h,1} \|_A^2 \\
&\lesssim \| \epb^{I,1} - \epb^{I,0} \|_{\laminv}^2 + (\lap t )^2 \LRp{ \| \Ib_{3}^0 \|_{S}^2 + \| I_4^0 \|_{\laminv}^2 } \\
&\quad + \lap t \| \epb^{h,0} \|_A^2 + \| e_{\ub}^0 \|_{\Vb}^2 + \| \ept^0 \|_{\Qt}^2 + \| \alb \cdot \epb^0 \|_{\laminv}^2 \\
&\lesssim (\lap t)^2 (h^{2k_{\pb}} \| \pb \|_{W^{1,\infty}(0, t_1; H^{k_{\pb} +1})}^2 + h^{2k_{\ub}} \| \pt \|_{W^{1,\infty}(0, t_1; H^{k_{\ub}})}^2) \\
&\quad  + (\lap t)^6 \| \pb, \pt \|_{W^{3,\infty}(0,t_1; L^2)}^2  + \lap t h^{2k_{p}} \| \pb^0 \|_{H^{k_p}}^2 + h^{2 k_{\ub}} ( \| \ub^0 \|_{k_{\ub}+1}^2 + \| \pt^0 \|_{k_{\ub}}^2 ) \\
&\lesssim (\lap t)^6 +  h^{2k_{\ub}} + h^{2k_{p}} ,
}
so \eqref{eq:initD-estm} is proved. % for all terms except $\| \Dpt^0 \|_{\Qt}$. 

To estimate $\| \eu^{h,1} \|_{\Vb}$, $\| \ept^{h,1} \|_{\Qt}$, $\| \epb^{h,1} \|_S$, we need another equation which is the sum of \eqref{eq:n0-eq1} and \eqref{eq:n1-eq1}. If we take $\vb = \eu^{h,0} + \eu^{h,1}$ in this equation, $\qt = - \ept^{h,0} - \ept^{h,1}$ in \eqref{eq:tmp1},   $\qii = - \epi^{h,0} - \epi^{h,1}$ in \eqref{eq:tmp3} for $1\le i \le N$, and add these equations altogether, then we get
\mltlns{
\| \eu^{h,1} \|_{\Vb}^2 - \| \eu^{h,0} \|_{\Vb}^2 + \| \ept^{h,1} - \alb \cdot \epb^{h,1} \|_{\laminv}^2 - \| \ept^{h,0} - \alb \cdot \epb^{h,0} \|_{\laminv}^2  \\
+ \sum_{i=1}^N \LRp{\| \epi^{h,1} \|_{s_i}^2 - \| \epi^{h,0} \|_{s_i}^2 } + 2\lap t \| \epb^{h,\half} \|_A^2  \\
= -\LRa{ 2\mu \strain(\eu^0), \strain(\eu^{h,1} + \eu^{h,0}) } + \LRa{ \ept^0, \div (\eu^{h,1} + \eu^{h,0}) } + \LRa{ R_h^0, \Dpt^0 } \\
 + \LRa{\laminv \alb \cdot \LRp{ \epb^{I,1} + \epb^{I,0}  }, \Dpt^0 } \\
+ 2 \lap t \LRs{ S(\Ib_{3}^0, \epb^{h,1} + \epb^{h,0})  - \LRa{ \laminv I_4^0, \alb \cdot \LRp{\epb^{h,1} + \epb^{h,0} } } }.
}
If 
\mltln{ \label{eq:case1}
 \LRa{\laminv \alb \cdot \LRp{ \epb^{I,1} + \epb^{I,0} }, \Dpt^0 }  + \LRa{ R_h^0, \Dpt^0 } \\
 \ge - \LRa{ 2\mu \strain(\eu^0), \strain(\eu^{h,1} + \eu^{h,0}) } + \LRa{ \ept^0, \div (\eu^{h,1} + \eu^{h,0}) } \\
 + 2 \lap t \LRs{ S(\Ib_{3}^0, \epb^{h,1} + \epb^{h,0})  - \LRa{ \laminv I_4^0, \alb \cdot \LRp{\epb^{h,1} + \epb^{h,0} } } }
%  - \LRa{ \laminv I_4^l, \alb \cdot \LRp{\epb^{h,1} + \epb^{h,0} } }
}
is true, then 
\algns{
&\| \eu^{h,1} \|_{\Vb}^2 + \| \ept^{h,1} - \alb \cdot \epb^{h,1} \|_{\laminv}^2 + \sum_{i=1}^N \| \epi^{h,1} \|_{s_i}^2 + 2\lap t \| \epb^{h,\half} \|_A^2  \\
& \le \| \eu^{h,0} \|_{\Vb}^2 + \| \ept^{h,0} - \alb \cdot \epb^{h,0} \|_{\laminv}^2 + \sum_{i=1}^N \| \epi^{h,0} \|_{s_i}^2 +  \LRa{2 \laminv \alb \cdot \LRp{ \epb^{I,1} + \epb^{I,0} } + R_h^0 , \Dpt^{0} } , 
}
so \eqref{eq:init-estm} for $\| \eu^{h,1} \|_{\Vb}$ follows from \eqref{eq:init-h-approx}, the Cauchy--Schwarz inquality, and the estimate of $\| \Dpt^0 \|_{\laminv}$.
If \eqref{eq:case1} is not true, then we have 
\algns{
&\| \eu^{h,1} \|_{\Vb}^2 + \| \ept^{h,1} - \alb \cdot \epb^{h,1} \|_{\laminv}^2 + \sum_{i=1}^N \| \epi^{h,1} \|_{s_i}^2 \\
& - \LRp{ \| \eu^{h,0} \|_{\Vb}^2 + \| \ept^{h,0} - \alb \cdot \epb^{h,0} \|_{\laminv}^2 + \sum_{i=1}^N \| \epi^{h,0} \|_{s_i}^2 } + 2\lap t \| \epb^{h,\half} \|_A^2  \\
&\quad \le -2 \LRa{ 2\mu \strain(\eu^0), \strain(\eu^{h,1} + \eu^{h,0}) } + 2 \LRa{ \ept^0, \div (\eu^{h,1} + \eu^{h,0}) } \\
%&\qquad + 2\LRa{\laminv \alb \cdot \LRp{ \epb^{h,1} + \epb^{h,0} }, 2\Dpt^0 } \\
&\qquad + 4 \lap t \LRs{ S(\Ib_{3}^0, \epb^{h,1} + \epb^{h,0}) - \LRa{\laminv \alb \cdot \LRp{ \epb^{h,1} + \epb^{h,0} }, I_4^0 } }.
}
If we use the Cauchy--Schwarz inquality and divide both sides by 
\mltlns{
\LRp{ \| \eu^{h,1} \|_{\Vb}^2 + \| \ept^{h,1} - \alb \cdot \epb^{h,1} \|_{\laminv}^2 + \sum_{i=1}^N \| \epi^{h,1} \|_{s_i}^2  }^\half \\
+ \LRp{  \| \eu^{h,0} \|_{\Vb}^2 + \| \ept^{h,0} - \alb \cdot \epb^{h,0} \|_{\laminv}^2 + \sum_{i=1}^N \| \epi^{h,0} \|_{s_i}^2 }^\half ,
}
we can obtain 
\mltlns{
\LRp{ \| \eu^{h,1} \|_{\Vb}^2 + \| \ept^{h,1} - \alb \cdot \epb^{h,1} \|_{\laminv}^2 + \sum_{i=1}^N \| \epi^{h,1} \|_{s_i}^2  }^\half \\
\lesssim \LRp{  \| \eu^{h,0} \|_{\Vb}^2 + \| \ept^{h,0} - \alb \cdot \epb^{h,0} \|_{\laminv}^2 + \sum_{i=1}^N \| \epi^{h,0} \|_{s_i}^2 }^\half \\
+ \| \eu^0 \|_{\Vb} + \| \ept^0 \|_{\Qt} + \lap t (\| \Ib_3^0 \|_{S} + \| I_4^0 \|_0 ) .
}
Then \eqref{eq:init-estm} for $\| \eu^{h,1} \|_{\Vb}$ is proved. The estimate \eqref{eq:init-estm} for $\| \ept^{h,1} \|_{\Qt}$ can be obtained from \eqref{eq:n1-eq1} by taking $\vb \in \Vbh$ such that $\LRa{ \ept^{h,1}, \div \vb} = \| \ept^{h,1} \|_{\Qt}^2$ and $\| \vb \|_{\Vb} \lesssim \| \ept^{h,1} \|_{\Qt}$. Finally, the same estimate for $\| \epb^{h,1} \|_S$ follows from the same argument using the triangle inequality as before. 
\end{proof}

%To estimate $\| \alb \cdot \Dpb^{1} \|_{\laminv}^2$ and $\lap t \mc{A}(\epb^{h,2}, \epb^{h,2})$, we use a similar argument to 
%relevant equations, i.e., \eqref{eq:method2-err1}, \eqref{eq:method2-err3} with $n=1$ and \eqref{eq:Dpb-eq1}.
%Taking $\vb = \Dub^{1} + \delta \wb^1$, $\qt = -\Dpt^{1}$, $\qii = \Dpi^{1}$ for $1 \le i \le N$ in these equations, we get 
%\algns{
%&\| \Dub^1 \|_{\Vb}^2 + C \| \Dpt^1 \|_{\Qt}^2 + \| \Dpt^1 \|_{\laminv}^2 + \sum_{i=1}^N \| \Dpi^1 \|_{s_i}^2 + \| \alb \cdot \Dpb^1 \|_{\laminv}^2 + \half \lap t \mc{A}(\epb^{h,2}, \epb^{h,2} ) \\
%&\le \half \lap t \mc{A} (\epb^{h,1}, \epb^{h,1} ) + \LRa{\laminv \cdot (\Ib_1^1 - \Ib_1^0 + \epb^{I,2} - \epb^{I,1} + \Dpt^0 ), \Dpt^1 } \\
%&\quad + \lap t \LRp{ \sum_{i=1}^N \LRa{ s_i \Ib_{3,i}^1 , \Dpi^1} + \LRa{ \laminv (\alb \cdot \Ib_3^1 + I_4^1 ), \alb \cdot \Dpb^1 } } .
%}
%Therefore 
%\algns{
%&\| \Dub^1 \|_{\Vb}^2 + \| \Dpt^1 \|_{\laminv}^2 + \sum_{i=1}^N \| \Dpi^1 \|_{s_i}^2 + \| \alb \cdot \Dpb^1 \|_{\laminv}^2 + \lap t \mc{A}(\epb^{h,2}, \epb^{h,2} ) \\
%&\lesssim \| \epb^{I,2} - \epb^{I,1} \|_{\laminv}^2 + (\lap t )^2 \LRp{ \| \Ib_{3,i}^1 \|_{s_i}^2 + \| \alb \cdot \Ib_3^1 + I_4^1 \|_{\laminv}^2 } \\
%&\quad + \| \Dpt^0 \|_{\laminv}^2 + \lap t \mc{A} (\epb^{h,1}, \epb^{h,1} ) .
%}

\subsection{The first partitioned method (elasticity-then-diffusion)} \label{subsec:etd}

We here present the first partitioned method inspired by the differential form of \eqref{eq:mpet:tp2}, i.e., 
\algns{
- \div \dub - \laminv \dpt + \laminv \alb \cdot \dpb = 0.
}

{\bf Method 1} 
\medskip 

Suppose that $(\ubh^0, \pth^0, \pbh^0)$ and $(\ubh^1, \pth^1, \pbh^1)$ are provided by the monolithic numerical method described in the subsection \ref{subsec:prelim}.
\begin{itemize}
 \item[{\bf Step 1}] For $n \ge 1$, given $\ubh^n$, $\pth^n$, $\pbh^{n}$, $\pbh^{n-1}$, compute $(\ubh^{n+1}, \pth^{n+1})$ by
\algn{
\label{eq:method1-disc1} &\LRa{ 2\mu \strain(\ubh^{n+1}), \strain(\vb) } - \LRa{\pth^{n+1}, \div \vb} = \LRa{\fb^{n+1}, \vb} \quad \forall \vb \in \Vbh , \\
\label{eq:method1-disc2} &- \LRa{\Div (\ubh^{n+1} - \ubh^n), \qt} - \LRa{\laminv (\pth^{n+1} - \pth^n ), \qt} \\
\notag &\quad = - \LRa{\laminv \alb \cdot (\pbh^{n} - \pbh^{n-1}) , \qt} \quad  \forall \qt \in Q_{t,h} .
}
 \item[{\bf Step 2}] Compute $\pbh^{n+1}$ with $\pth^{n}$ and $\pth^{n+1}$ by
\algn{ \label{eq:method1-disc3}
&- \LRa{ s_i \frac{\piih^{n+1} - \piih^n }{\lap t} , \qii} \\
&\notag - \LRa{\alpha_i \laminv \alb \cdot \frac{\pbh^{n+1} - \pbh^n}{\lap t} + \xi_i \LRp{ \pbh^{n+\half} } , \qii } - \ahi \LRp{ \piih^{n+\half}  , \qii } \\
&\notag = - \LRa{ g^{n+\half} ,\qii} - \LRa{\alpha_i \laminv \frac{\pth^{n+1} - \pth^{n}}{\lap t}, \qii}
\quad \foralls \qii \in Q_{i,h} ,  1 \le i \le N.
}

 \item[{\bf Step 3}] Repeat {\bf Step 1} for $n \leftarrow n+1$
\end{itemize}

\bigskip 
%The first system \eqref{eq:method1-disc1}--\eqref{eq:method1-disc2} is the Lam\'e equation which is well-posed. 

\begin{theorem} \label{thm:method1} Let $(\ub, \pt, \pb)$ be the exact solution of \eqref{eq:mpet:vf} for compatible initial data $(\ub^0, \pt^0, \pb^0)$. For numerical initial data $(\ubh^0, \pth^0, \pbh^0)$ satisfying \eqref{eq:initial-assumption1}, suppose that $\{ (\ubh^n, \pth^n, \pbh^n) \}$ is a numerical solution obtained by {\bf Method 1}. Assuming that the exact solution is sufficiently regular, 
\algns{
\| \ub^n - \ubh^n \|_{\Vb} + \| \pt^n - \pth^n \|_{\Qt} + \| \pb^n - \pbh^n \|_S &\lesssim (\lap t)^2 + h^{k_{\ub}} + h^{k_p}, \\
\| \pb^n - \pbh^n \|_{A} &\lesssim (\lap t)^{\frac{3}{2}} + (\lap t)^{-\half} (h^{k_{\ub}} + h^{k_p} ) 
}
hold with implicit constants depending on the norms of the exact solution.
\end{theorem}
\begin{proof}
By the triangle inequality and the optimal approximation properties \eqref{eq:intp-estm1} and \eqref{eq:intp-estm2}, it suffices to estimate $\| \eu^{h,n} \|_{\Vb}$, $\| \ept^{h,n} \|_{\Qt}$, $\| \epb^{h,n} \|_S$, and $\| \epb^{h,n} \|_A$. To estimate these terms, 
we consider the differences of continuous equations and discrete equations. The three continuous equations are 
\eqref{eq:method1-conti1} with $n = l$, the difference of \eqref{eq:method1-conti2} with $n=l+1$ and $n=l$, and \eqref{eq:method1-conti3} with $n=l$. 
If we subtract the discrete equations \eqref{eq:method1-disc1}, \eqref{eq:method1-disc2}, \eqref{eq:method1-disc3} with $n=l$ from these equations, assuming that $\pb$ is sufficiently regular in case for the DG or EG methods, the differences of the continuous equations and the discrete equations are 
\algns{
& \LRa{ 2\mu \strain(\eu^{l+1}), \strain(\vb) } - \LRa{\ept^{l+1}, \div \vb} = 0  \quad \forall \vb \in \Vbh , \\
& - \LRa{\Div \LRp{\eu^{l+1} - \eu^l } , \qt} - \LRa{\laminv \LRp{\ept^{l+1} - \ept^l} , \qt} \\
&\quad = - \LRa{\laminv \alb \cdot \LRp{ \pb^{l+1} - \pb^l - (\pbh^{l} - \pbh^{l-1}) }, \qt} \quad \forall \qt \in Q_{t,h} , \\
& -\LRa{ s_i \LRp{ \dpii^{l+\half}  -  \frac{\piih^{l+1} - \piih^l }{\lap t} } , \qii} \\
& \quad - \LRa{\alpha_i \laminv \alb \cdot \LRp{ \dpb^{l+\half} - \frac{\pbh^{l+1} - \pbh^l}{\lap t} }
+ \xi_i \LRp{ \epb^{l+\half} } , \qii } - \ahi \LRp{\epi^{l+\half} , \qii }  \\
&\quad = - \LRa{\alpha_i \laminv \LRp{ \dpt^{l+\half} - \frac{\pth^{l+1} - \pth^{l}}{\lap t} }, \qii}
\qquad \foralls \qii \in Q_{i,h} ,  1 \le i \le N
}
for $l \ge 1$. These equations with the interpolations $\Pi_h^{\Vb}$, $\Pi_h^{\Qt}$, $\Pi_h^{\Qb}$, lead to the error equations
\begin{subequations}
\label{eq:method1-err-a}
\algn{
\label{eq:method1-err1a} & \LRa{ 2\mu \strain(\eu^{h,l+1}), \strain(\vb) } - \LRa{\ept^{h,l+1}, \div \vb} = 0 , \\
\label{eq:method1-err2a} 
&-\LRa{\Div \Dub^{l} , \qt} - \LRa{\laminv \Dpt^{l}, \qt} \\
&\notag \qquad = - \LRa{\laminv \alb \cdot (\Ib_1^l - \Ib_1^{l-1} + \Ib_2^l - \Ib_2^{l-1} + \Dpb^{l-1} ), \qt} , \\
\label{eq:method1-err3a} & -\LRa{ s_i \Dpi^l , \qii} -  \LRa{ \alpha_i \laminv (\alb \cdot \Dpb^l )
+ \lap t \xi_i \LRp{ \epb^{h,l+\half} } , \qii } - \lap t \ahi \LRp{ \epi^{h,l+\half} , \qii } \\
&\notag \qquad = \lap t \LRa{s_i \Ib_{3,i}^l + \alpha_i \laminv \alb \cdot \Ib_3^l - \alpha_i \laminv I_4^l, \qii} - \LRa{ \alpha_i \laminv \Dpt^l, \qii }
}
\end{subequations}
for $l \ge 1$, $i=1,\ldots, N$ where $\Ib_{j}^l = (I_{j,1}^l, \ldots, I_{j,N}^l)$ are defined as 
\algns{
\Ib_{1,i}^l &= \pii^{l+1} - \pii^l , & \Ib_{2,i}^l &= \pii^{l}  - \Pi_h^{\Qi} \pii^{l} 
}
and in \eqref{eq:I3-I4}. Considering the differences of two consecutive time steps of \eqref{eq:method1-err1a}, we get 
%\subeqns{eq:method1-err}{
\algn{
\label{eq:method1-err1} 
&\LRa{ 2\mu \strain(\Dub^{l} ), \strain(\vb) } - \LRa{\Dpt^{l} , \div \vb} = 0 
}
%\label{eq:method1-err2} 
%&-\LRa{\Div \Dub^{l} , \qt} - \LRa{\laminv \Dpt^{l}, \qt} \\
%&\notag \qquad = - \LRa{\laminv \alb \cdot (\Ib_1^l - \Ib_1^{l-1} + \Ib_2^l - \Ib_2^{l-1} + \Dpb^{l-1} ), \qt} \\
%\label{eq:method1-err3} 
%& \LRa{ s_i \Dpi^{l} , \qii} + \LRa{\alpha_i \laminv (-\Dpt^l + \alb \cdot \Dpb^l ) + \lap t \xi_i \LRp{ \epb^{h,l+\half} } , \qii } + \lap t \ahi \LRp{\epi^{h,l+\half} , \qii } \\
%&\notag \qquad = \lap t\LRp{ \LRa{ s_i \Ib_{3,i}^l, \qii} + \LRa{ \alpha_i \laminv (\alb \cdot \Ib_3^l - I_4^l), \qii} } , \quad  1 \le i \le N
%}
for $l \ge 1$.

We are ready to prove the error estimates. Since the proof is long, we split the proof into two steps. 
In the first step, we prove 

\mltln{ \label{eq:l2-D-estm}
\sum_{l=1}^{n-1} \LRp{ \| \Dub^{l} \|_{\Vb}^2 + \| \Dpt^{l} \|_{\Qt}^2 + \| \Dpb^{l} \|_{S}^2 } + \lap t \| \epb^{h,n} \|_A^2 \\
\lesssim (\lap t)^4 + (\lap t)^2 h^{2k_{\ub}} + \lap t h^{2k_{\pb}} .
}
In the second step, we prove 
\algn{ \label{eq:l2-estm}
\| \eu^{h,n} \|_{\Vb} + \| \ept^{h,n} \|_{\Qt} + \| \epb^{h,n} \|_{S} \lesssim  (\lap t)^2 + h^{k_p} + h^{k_{\ub}} .
}
It is easy to see that these two estimates complete the proof.

%\noindent {\bf Estimate of $ \sum_{l=1}^n ( \| \Dub^l \|_{\Vb}^2 + \| \Dpt^l \|_{\Qt}^2 + \| \Dpb^l \|_S^2 )+ \lap t \| \epb^{h,n} \|_{A}^2$} 
%Here we begin $ \sum_{l=1}^n ( \| \Dub^l \|_{\Vb}^2 + \| \Dpt^l \|_{\Qt}^2 + \| \Dpb^l \|_S^2 )+ \lap t \| \epb^{h,n} \|_{A}^2$. 
For the first step let $\wb^l$ be an element in $\Vbh$ satisfying the condition \eqref{eq:inf-sup} for $\Dpt^l$. 
If we take $\vb = \Dub^l + \delta \wb^{l}$ in \eqref{eq:method1-err1}, $\qt = - \Dpt^l$ with $\delta >0$ as above in \eqref{eq:method1-err2a}, $\qii = -\Dpi^l$ in \eqref{eq:method1-err3a}, and add them altogether, then the sum gives 
\algns{
%&\| \Dub^{l} \|_{\Vb}^2 + \delta \LRa{2 \mu \e(\Dub^l), \strain(\wb^l) } - \delta \LRa{ \Dpt^l, \div \wb^l } + \| \Dpt^l \|_{\laminv}^2 + \| \Dpb^l \|_{S}^2  + \half \lap t \| \epb^{h,l+1} \|_A^2 \\
& \LRa{ 2 \mu \strain(\Dub^{l}), \strain(\Dub^l + \delta \wb^l) } - \LRa{\Dpt^l, \div (\Dub^l + \delta \wb^l) } + \LRa{ \div \Dub^l, \Dpt^l } \\
&\quad + \| \Dpt^l \|_{\laminv}^2 + \| \Dpb^l \|_{S}^2  + \half \lap t \| \epb^{h,l+1} \|_A^2 \\
&\quad \le   \LRa{\laminv \alb \cdot (\Ib_1^l - \Ib_1^{l-1} + \Ib_2^l - \Ib_2^{l-1} + \Dpb^l + \Dpb^{l-1} ), \Dpt^l} \\
&\qquad + \half \lap t \| \epb^{h,l} \|_A^2 - \lap t \LRp{ S (\Ib_{3}^l , \Dpb^l) + \LRa{ \laminv I_4^l, \alb \cdot \Dpb^l}  } 
%&\qquad - \delta_{n1} \LRp{  \LRa{2\mu \e(\eu^{I,0}), \strain(\Dub^1 + \delta \wb^{1})} - \LRa{\ept^{I,0}, \div (\Dub^1 + \delta \wb^{1})} + \LRa{\div \eu^{I,0}, \Dpt^1} + \LRa{\laminv \ept^{I,0}, \Dpt^1} }
}
for $l \ge 1$. By Young's inequality, 
\algns{
&\LRa{\laminv \alb \cdot (\Ib_1^l - \Ib_1^{l-1} + \Ib_2^l - \Ib_2^{l-1} + \Dpb^l + \Dpb^{l-1} ), \Dpt^l} \\
&\qquad - \lap t \LRp{ S (\Ib_{3}^l , \Dpb^l) + \LRa{ \laminv I_4^l, \alb \cdot \Dpb^l}  } \\
&\le \frac 1{2\e_1} \| \alb \cdot \LRp{\Ib_1^{l} - \Ib_1^{l-1} + \Ib_2^{l} - \Ib_2^{l-1} } \|_{\laminv}^2 \\
&\quad + \frac 12 \| \alb \cdot \Dpb^{l} \|_{\laminv}^2+ \frac 1{2 + \e_2} \| \alb \cdot \Dpb^{l-1} \|_{\laminv}^2 + \LRp{\half  + \frac{\e_1}2 + \frac{2+\e_2}4 } \| \Dpt^{l} \|_{\laminv}^2 \\
%&\quad  + \half \lap t \LRp{ \mc{A} (\epb^{h,n+1}, \epb^{h,n+1}) + \mc{A} (\epb^{h,n}, \epb^{h,n}) } \\
&\quad + \frac{(\lap t)^2}{2\e_3} \LRp{ \| \Ib_{3}^l \|_{S}^2 + \| I_4^l \|_{\laminv}^2 } + \e_3 \| \Dpb^l \|_{S}^2  
}
with any $\e_1, \e_2, \e_3 >0$. 
From \eqref{eq:stokes-stability} and the above two inequalities we can obtain 
\algns{
& C_0 \| \Dub^{l} \|_{\Vb}^2 + {C_0} \| \Dpt^l \|_{\Qt}^2 +  \| \Dpb^l \|_{S}^2  + \half \lap t \| \epb^{h,l+1} \|_A^2 \\
&\quad \le  \frac 1{2\e_1} \| \alb \cdot \LRp{\Ib_1^{l} - \Ib_1^{l-1} + \Ib_2^{l} - \Ib_2^{l-1} } \|_{\laminv}^2 +  \frac 1{2 + \e_2} \| \alb \cdot \Dpb^{l-1} \|_{\laminv}^2  \\
&\qquad + \LRp{\frac{\e_1}2 + \frac{\e_2}4 } \| \Dpt^{l} \|_{\laminv}^2 + \frac{(\lap t)^2}{2\e_3} \LRp{ \| \Ib_{3}^l \|_{S}^2 + \| I_4^l \|_{\laminv}^2 } + \e_3 \| \Dpb^l \|_{S}^2 .
%&\qquad - \delta_{n1} \LRp{  \LRa{2\mu \e(\eu^{I,0}), \e(\Dub^1 + \delta \wb^{1})} - \LRa{\ept^{I,0}, \div (\Dub^1 + \delta \wb^{1})} + \LRa{\div \eu^{I,0}, \Dpt^1} + \LRa{\laminv \ept^{I,0}, \Dpt^1} }
}
We choose sufficiently small $\e_1$ and $\e_2$ such that $\LRp{\frac{\e_1}2 + \frac{\e_2}4 } \| \qt \|_{\laminv}^2 \le \frac{C_0}2 \| \qt \|_{\Qt}^2$ for all $\qt \in \Qth$ and for the $C_0$ in \eqref{eq:stokes-stability}. 
Then the above inequality gives 
\algns{
& C_0 \| \Dub^{l} \|_{\Vb}^2 + \frac{C_0}2 \| \Dpt^l \|_{\Qt}^2 + \LRp{1 - \e_3} \| \Dpb^l \|_{S}^2  + \half \lap t \| \epb^{h,l+1} \|_A^2 \\
&\quad \le  \frac 1{2\e_1} \| \alb \cdot \LRp{\Ib_1^{l} - \Ib_1^{l-1} + \Ib_2^{l} - \Ib_2^{l-1} } \|_{\laminv}^2 +  \frac 1{2 + \e_2} \| \alb \cdot \Dpb^{l-1} \|_{\laminv}^2  \\
&\qquad  + \frac{(\lap t)^2}{2\e_3} \LRp{ \| \Ib_{3}^l \|_{S}^2 + \| I_4^l \|_{\laminv}^2 }  .
%&\qquad - \delta_{n1} \LRp{  \LRa{2\mu \e(\eu^{I,0}), \e(\Dub^1 + \delta \wb^{1})} - \LRa{\ept^{I,0}, \div (\Dub^1 + \delta \wb^{1})} + \LRa{\div \eu^{I,0}, \Dpt^1} + \LRa{\laminv \ept^{I,0}, \Dpt^1} }
}
We take $\e_3$ to satisfy $ 1 - \e_3  - \frac{1}{2+\e_2} = \e_3$. %$1 - 2\e_3 = \half + \frac 1{1+\e_2}$. 
Taking the summation of the above inequality over the index $l$ from 1 to $n-1$ and ignoring some nonnegative terms, we can get 
\algns{
&\sum_{l=1}^{n-1} \LRp{ C_0 \| \Dub^{l} \|_{\Vb}^2 + \frac{C_0}2 \| \Dpt^{l} \|_{\Qt}^2 + \e_3 \| \Dpb^{l} \|_{S}^2 } + \half \lap t \| \epb^{h,n} \|_A^2 \\
&\quad \le  \sum_{l=1}^{n-1} \frac 1{2\e_1} \| \alb \cdot \LRp{\Ib_1^{l} - \Ib_1^{l-1} + \Ib_2^{l} - \Ib_2^{l-1} } \|_{\laminv}^2 + \frac 1{2+\e_2} \| \alb \cdot \Dpb^{0} \|_{\laminv}^2 \\
&\qquad  + \half \lap t \| \epb^{h,1} \|_A^2  + \frac{(\lap t)^2}{2\e_3} \sum_{l=1}^{n-1} \LRp{ \| \Ib_{3}^{l} \|_{S}^2 + \| I_4^{l} \|_{\laminv}^2 } . %\\
%&\qquad - \LRp{  \LRa{2\mu \e(\eu^{I,0}), \e(\Dub^1 + \delta \wb^{1})} - \LRa{\ept^{I,0}, \div (\Dub^1 + \delta \wb^{1})} + \LRa{\div \eu^{I,0}, \Dpt^1} + \LRa{\laminv \ept^{I,0}, \Dpt^1} } .
}
Since $\e_1$, $\e_2$, $\e_3$ depend only on $C_0$, they are independent of the mesh sizes. A standard argument gives 
\algn{
\label{eq:I1-diff-estm} \| \alb \cdot (\Ib_1^l - \Ib_1^{l-1}) \|_{\laminv} &\lesssim (\lap t)^2 \| \alb \cdot \pb \|_{W^{2,\infty}(t_{l-1}, t_{l+1}; L^2)}, \\
\label{eq:I2-diff-estm} \| \alb \cdot (\Ib_2^l - \Ib_2^{l-1}) \|_{\laminv} &\lesssim h^{k_p} \| \alb \cdot \pb \|_{L^{\infty}(t_{l-1}, t_{l}; H^{k_p+1})} .
}
From these, \eqref{eq:I3-estm}, \eqref{eq:I4-estm}, \eqref{eq:initD-estm}, and the assumption \eqref{eq:initial-assumption1}, we can obtain %\eqref{eq:l2-D-estm}.
\mltln{ \label{eq:method1-D-estm}
\sum_{l=1}^{n-1} \LRp{ \| \Dub^{l} \|_{\Vb}^2 + \| \Dpt^{l} \|_{\Qt}^2 + \| \Dpb^{l} \|_{S}^2 } + \lap t \| \epb^{h,n} \|_A^2 \\
 \lesssim (\lap t)^4 \| \alb \cdot \pb, \pt \|_{W^{2,\infty}(0, t_n; L^2)}^2 + h^{2 k_p} \LRp{ \| \pb \|_{L^\infty(0, t_n; H^{k_p+1})}^2 + \| \pb^0 \|_{H^{k_{\ub}}}^2 } \\
+ h^{2k_{\ub}} \LRp{ \| \pt^0 \|_{H^{k_{\ub}}}^2 + \| \ub^0 \|_{H^{k_{\ub}+1}}^2 } .
}

%\noindent {\bf Estimate of $\| \eu^{h,n} \|_{\Vb} + \| \epb^{h,n} \|_{\Qt} + \| \epb^{h,n} \|_S$}
Now we begin the second step and prove an estimate of $\| \eu^{h,n} \|_{\Vb} + \| \epb^{h,n} \|_{\Qt} + \| \epb^{h,n} \|_S$.
%$\sum_{i=1}^N \| \epi^{h,n+1} \|_{s_i}^2 + \| \alb \cdot \epb^{h,n+1} \|_{\laminv}^2$.
%Note that it is sufficient to estimate this in the case that 
%$\sum_{i=1}^N \| \epi^{h,l} \|_{s_i}^2 + \| \alb \cdot \epb^{h,l} \|_{\laminv}^2$
%is monotonely increasing for $1 \le l \le n+1$. 
%In particular, 
%\algn{ \label{eq:max-assumption}
%\sum_{i=1}^N \| \epi^{h,n+1} \|_{s_i}^2 + \| \alb \cdot \epb^{h,n+1} \|_{\laminv}^2 \ge \sum_{i=1}^N \| \epi^{h,n} \|_{s_i}^2 + \| \alb \cdot \epb^{h,n} \|_{\laminv}^2 .
%}
To do it we add the following three equations: 
\begin{itemize}
 \item[1.] the sum of \eqref{eq:method1-err1a} of indices $l+1$ and $l$ with $\vb = \Dub^l$
 \item[2.] \eqref{eq:method1-err2a} with $\qt = -2\ept^{h,l+\half}$
 \item[3.] the equations \eqref{eq:method1-err3a} with $\qii = - 2 \epi^{h,l+\half}$ for $1 \le i \le N$.
\end{itemize}
From the sum of these equations we get 
\algn{ \label{eq:Xl-form}
&\| \eu^{h,l+1} \|_{\Vb}^2 - \| \eu^{h,l} \|_{\Vb}^2 + \| \alb \cdot \epb^{h,l+1} - \ept^{h,l+1} \|_{\laminv}^2 - \| \alb \cdot \epb^{h,l} - \ept^{h,l} \|_{\laminv}^2 \\
&\notag  + \sum_{i=1}^N \LRp{ \| \epi^{h,l+1} \|_{s_i}^2 - \| \epi^{h,l} \|_{s_i}^2 } + 2\lap t \| \epb^{h,l+\half} \|_A^2  \\
&\notag \quad = 2 \LRa{\laminv \alb \cdot \LRp{ \Ib_1^l - \Ib_1^{l-1} + \Ib_2^l - \Ib_2^{l-1} }, \ept^{h, l+\half} } \\
&\notag \qquad + 2 \lap t \LRs{ S(\Ib_{3}^l, \epb^{h,l+\half})  + \LRa{ \laminv I_4^l, \alb \cdot \epb^{h,l+\half} } }   \\
&\notag \qquad - 2 \LRa{\laminv  \alb \cdot ( \Dpb^{l} - \Dpb^{l-1} ), \ept^{h, l+\half} } . %\\
}
%%%%%%%%%%%%%%%%%%%%%%%%%%%%%%%%%%%%%%%%%%%%%%%%%%%%%%%%%%%%%%%%%%%%%%
%-----
%
%Note that 
%\algns{
%& 2 \sum_{l=1}^{n-1} \LRa{\laminv  \alb \cdot ( \Dpb^{l} - \Dpb^{l-1} ), \ept^{h, l+\half} } \\
%&= \LRa{\laminv \alb \cdot (\Dpb^1 - \Dpb^0), \ept^{h,2} + \ept^{h,1} } + \LRa{\laminv \alb \cdot (\Dpb^2 - \Dpb^1), \ept^{h,3} + \ept^{h,2} } \\
%&\quad + \cdots + \LRa{\laminv \alb \cdot (\Dpb^{n-1} - \Dpb^{n-2}), \ept^{h,n} + \ept^{h,n-1} } \\
%&= - \LRa{ \laminv \alb \cdot \Dpb^0, \ept^{h,2} + \ept^{h,1} } - \LRa{ \laminv \alb \cdot \Dpb^1 , \Dpt^2} \\
%&\quad - \cdots - \LRa{ \laminv \alb \cdot \Dpb^{n-2}, \Dpt^{n-1}} - \LRa{ \laminv \alb \cdot \Dpb^{n-1}, \ept^{h,n} + \ept^{h,n-1} } \\
%&= - \LRa{ \laminv \alb \cdot \Dpb^0, \ept^{h,2} + \ept^{h,1} } - \LRa{ \laminv \alb \cdot \Dpb^{n-1}, \ept^{h,n} + \ept^{h,n-1} } \\
%&\quad - \sum_{l=1}^{n-2} \LRa{ \laminv \alb \cdot \Dpb^{l}, \Dpt^{l+1} } .
%}
%
%-----
%%%%%%%%%%%%%%%%%%%%%%%%%%%%%%%%%%%%%%%%%%%%%%%%%%%%%%%%%%%%%%%%%%%%%%
Defining $X_l, Y_l \ge 0$ as 
\algn{
\label{eq:Xn} X_l^2 &= \| \eu^{h,l} \|_{\Vb}^2 + \| \alb \cdot \epb^{h,l} - \ept^{h,l} \|_{\laminv}^2 + \sum_{i=1}^N \| \epi^{h,l} \|_{s_i}^2 , \\
\label{eq:Yn} Y_l &= \sqrt{\lap t} \| \epb^{h,l+\half} \|_A ,
}
the left-hand side of the above equality is $X_{l+1}^2 - X_l^2 + Y_l^2$. 
%\algns{
%X_{n+1}^2 - X_n^2 + Y_n^2 &=  2 \lap t \LRs{ \LRp{ S \LRp{ \Ib_{3}^n, \epi^{h,n+\half} }  + \LRa{ \laminv  I_4^n, \alb \cdot \epb^{h,n+\half} } }  } \\
%& \quad + \LRa{\laminv \Dpt^n  , \alb \cdot ( \Dpb^{n} + \Dpb^{n-1} ) } .
%}
%Noting that $\epb^{h,l+1} - \epb^{h,l-1} = \Dpb^l + \Dpb^{l-1}$, 
Noting that we can obtain $\| \ept^{h,l} \|_{\Qt} \lesssim \| \eu^{h,l} \|_{\Vb}$ from \eqref{eq:method1-err1a}, we can also show that 
\algn{ \label{eq:max-form}
\max \{ \| \ept^{h,l} \|_{\laminv}, \| \epb^{h,l} \|_S \} \le C_2 X_l
}
by the triangle inequality and the definition of $\| \cdot \|_S$ with some $C_2>0$ independent of $h$. If we use the formula  
\algns{
& - 2 \sum_{l=1}^{n-1} \LRa{\laminv  \alb \cdot ( \Dpb^{l} - \Dpb^{l-1} ), \ept^{h, l+\half} } \\
&= 2 \LRa{ \laminv \alb \cdot \Dpb^0, \ept^{h,\frac 32} } + 2 \LRa{ \laminv \alb \cdot \Dpb^{n-1}, \ept^{h,n-\half} }  + \sum_{l=1}^{n-2} \LRa{ \laminv \alb \cdot \Dpb^{l}, \Dpt^{l+1} } 
}
obtained by a summation by parts argument, taking the summation of \eqref{eq:Xl-form} for $l$ over $1$ to $n-1$ gives 
\algns{
&X_{n}^2 + \sum_{l=1}^{n-1} Y_l^2 \\
&= X_1^2 + 2 \sum_{l=1}^{n-1} \LRa{\laminv \alb \cdot \LRp{ \Ib_1^l - \Ib_1^{l-1} + \Ib_2^l - \Ib_2^{l-1} }, \ept^{h, l+\half} }  \\
& \quad + 2 \lap t \sum_{l=1}^{n-1} \LRs{ \LRp{ S\LRp{ \Ib_{3}^l, \epb^{h,l+\half} } + \LRa{ \laminv I_4^l, \alb \cdot \epb^{h,l+\half} } }  } \\
& \quad + 2 \LRa{ \laminv \alb \cdot \Dpb^0, \ept^{h,\frac 32} } + 2 \LRa{ \laminv \alb \cdot \Dpb^{n-1}, \ept^{h,n-\half} }  + \sum_{l=1}^{n-2} \LRa{ \laminv \alb \cdot \Dpb^{l}, \Dpt^{l+1} } . 
%&\quad + \delta_{n1} \LRp{ 2 \LRa{2\mu \e(\eu^{I,0}), \e(\eu^{h,\half})} - 2\LRa{\ept^{I,0}, \div \eu^{h,\half} } - \LRa{\div \eu^{I,0}, \Dpt^0} - \LRa{\laminv \ept^{I,0}, \Dpt^0} } .
}
If $X_n < \max_{1 \le l \le n} X_l$, then it suffices to estimate $X_{n_0}$ for the smallest $1 \le n_0 < n$ such that $X_{n_0} = \max_{1 \le l \le n_0} X_l$ because the estimate of $X_{n_0}$ ($>X_n$) will also give an estimate of $X_n$ which can be used to obtain \eqref{eq:l2-estm}.
Therefore we will show an estimate of $X_n$ below with the assumption $X_n = \max_{1 \le l \le n} X_l$.

By the Cauchy--Schwarz inequality and \eqref{eq:max-form}, 
\algns{
X_{n}^2 + \sum_{l=1}^{n-1} Y_l^2 &\le 2 C_2 \sum_{l=1}^{n-1} \| \alb \cdot \LRp{ \Ib_1^l - \Ib_1^{l-1} + \Ib_2^l - \Ib_2^{l-1} } \|_{\laminv} X_n \\
& \quad + \LRp{ 2 \lap t C_2 \sum_{l=1}^{n-1} \LRs{ \| \Ib_{3}^l \|_{S} + \| I_4^l \|_{\laminv} }  } X_n \\
& \quad + 2 C_2 \LRp{\| \alb \cdot \Dpb^0 \|_{\laminv} + \| \alb \cdot \Dpb^{n-1} \|_{\laminv} } X_n \\ 
& \quad + X_1^2 + \sum_{l=1}^{n-2} \LRs{ \| \Dpt^{l+1} \|_{\laminv}^2  + \| \alb \cdot \Dpb^{l} \|_{\laminv}^2 } .
%+ \| \alb \cdot \Dpb^0 \|_{\laminv}^2 + \| \eu^{I,0} \|_{\Vb}^2 \\
%& \quad + \| \ept^{I,0} \|_{\laminv}^2 + \| \Dpt^0 \|_{0}^2 .
}
This inequality has a form of $A^2 + B^2 \le CA + D^2$ with $A = X_n$, $B = \sum_{l=1}^{n-1} Y_l$,
\algns{
C &= 2 C_2 \sum_{l=1}^{n-1} \LRs{\| \alb\cdot \LRp{\Ib_1^l - \Ib_1^{l-1} + \Ib_2^l - \Ib_2^{l-1} } \|_{\laminv} + \lap t \LRp{ \| \Ib_3^l \|_S + \| I_4^l \|_{\laminv} } } \\
&\quad  + 2 C_2 \LRp{\| \alb \cdot \Dpb^0 \|_{\laminv} + \| \alb \cdot \Dpb^{n-1} \|_{\laminv} } , \\
D &=  X_1^2 + \sum_{l=1}^{n-2} \LRs{ \| \Dpt^{l+1} \|_{\laminv}^2  + \| \alb \cdot \Dpb^{l} \|_{\laminv}^2 } .
}
We may assume $A, B > 0$ without loss of generality. Then it is easy to show that the above inequality implies
\algns{
\text{either} \quad A + B \le 4C \qquad \text{or} \qquad A + B \le 2 \sqrt{D} .
}
Therefore either 
\mltlns{
X_n + \LRp{ \sum_{l=1}^{n-1} Y_l^2 }^\half \\
\le 8 C_2 \sum_{l=1}^{n-1} \LRs{\| \alb\cdot \LRp{\Ib_1^l - \Ib_1^{l-1} + \Ib_2^l - \Ib_2^{l-1} } \|_{\laminv} + \lap t \LRp{ \| \Ib_3^l \|_S + \| I_4^l \|_{\laminv} } } \\
\quad  + 8 C_2 \LRp{\| \alb \cdot \Dpb^0 \|_{\laminv} + \| \alb \cdot \Dpb^{n-1} \|_{\laminv} }
}
or 
\algns{
X_n + \LRp{ \sum_{l=1}^{n-1} Y_l^2 }^\half 
\le 2 \LRp{ X_1^2 + \sum_{l=1}^{n-2} \LRs{ \| \Dpt^{l+1} \|_{\laminv}^2  + \| \alb \cdot \Dpb^{l} \|_{\laminv}^2 } }^\half  
}
holds. Recall that \eqref{eq:init-estm} gives an estimate of $X_1$. Then from the previous estimates \eqref{eq:I1-diff-estm}, \eqref{eq:I2-diff-estm}, \eqref{eq:I3-estm}, \eqref{eq:I4-estm}, and \eqref{eq:method1-D-estm}, we can conclude that 
\algns{
X_n + \LRp{ \sum_{l=1}^{n-1} Y_l^2 }^\half \lesssim (\lap t)^2 + h^{k_{\ub}} + h^{k_{p}} .
}
This proves \eqref{eq:l2-estm} for $\| \eu^{h,n} \|_{\Vb}$. The estimate for other two terms easily follows because $\| \epb^{h,n} \|_S \le C_2 X_n$ and $\| \ept^{h,n} \|_{\Qt} \lesssim \| \eu^{h,n} \|_{\Vb}$ hold.

%Since either 
%\algns{
%&2 \lap t \sum_{l=1}^n \sum_{i=1}^N \LRs{ \LRp{ \LRa{ s_i \Ib_{3,i}^l, \epi^{h,l+\half} } + \LRa{ \alpha_i \laminv (\alb \cdot \Ib_3^l + I_4^l), \epi^{h,l+\half} } }  } \\
%&> \sum_{l=1}^n \LRa{\laminv \Dpt^l  , \alb \cdot ( \Dpb^{l} + \Dpb^{l-1} ) }
%}
%or 
%\algns{
%&2 \lap t \sum_{l=1}^n \sum_{i=1}^N \LRs{ \LRp{ \LRa{ s_i \Ib_{3,i}^l, \epi^{h,l+\half} } + \LRa{ \alpha_i \laminv (\alb \cdot \Ib_3^l + I_4^l), \epi^{h,l+\half} } }  } \\
%&\le \sum_{l=1}^n \LRa{\laminv \Dpt^l  , \alb \cdot ( \Dpb^{l} + \Dpb^{l-1} ) } ,
%}
%we have either
%\algn{ \label{eq:Xn-case1}
%X_{n}^2 &\le X_1^2 + 4 \lap t \sum_{l=1}^n \sum_{i=1}^N \LRs{ \LRp{ \LRa{ s_i \Ib_{3,i}^l, \epi^{h,l+\half} } + \LRa{ \alpha_i \laminv (\alb \cdot \Ib_3^l + I_4^l), \epi^{h,l+\half} } }  } 
%}
%or 
%\algn{ \label{eq:Xn-case2}
%X_{n}^2 &\le X_1^2 + 2\sum_{l=1}^n \LRa{\laminv \Dpt^l  , \alb \cdot ( \Dpb^{l} + \Dpb^{l-1} ) } .
%}
%For \eqref{eq:Xn-case1}, from the maximality assumption of $X_n$, we have
%\algns{
%X_n \le X_1 + 4 \lap t \sum_{l=1}^n \LRs{ \LRp{ \sum_{i=1}^N \| \Ib_{3,i}^l \|_{s_i}^2 }^\half + \| \alb \cdot \Ib_3^l + I_4^l \|_{\laminv} } .
%} 
%For \eqref{eq:Xn-case2}, the arithmetic-geometric mean inequality and \eqref{eq:l2-D-estm} gives 
%\algns{
%X_n \le X_1 + (\lap t)^2 + (\lap t) h^{k_{\ub}} + \sqrt{\lap t} h^{k_{\pb}} .
%}

\end{proof}

\subsection{The second partitioned algorithm (diffusion-then-elasticity)} \label{subsec:dte}

We now present the second partitioned method. 

\medskip
{\bf Method 2}

Suppose that $(\ubh^0, \pth^0, \pbh^0)$ and $(\ubh^1, \pth^1, \pbh^1)$ are provided by the monolithic numerical method described in the subsection \ref{subsec:prelim}.
\begin{itemize}
 \item[{\bf Step 1}] For given $\pth^n$, $\pth^{n-1}$, compute $\pbh^{n+1}$ from 
\algn{ \label{eq:method2-disc1}
&- \LRa{ s_i \frac{\piih^{n+1} - \piih^n }{\lap t} , \qii} \\
&\notag - \LRa{\alpha_i \laminv \alb \cdot \frac{\pbh^{n+1} - \pbh^n}{\lap t} + \xi_i \LRp{ \pbh^{n+\half} } , \qii } - \ahi \LRp{ \piih^{n+\half} , \qii } \\
&\notag = - \LRa{ g^{n+\half}, \qii} - \LRa{\alpha_i \laminv \frac{\pth^n - \pth^{n-1}}{\lap t}, \qii}
\quad \foralls \qii \in Q_{i,h} ,  1 \le i \le N.
}
 \item[{\bf Step 2}] For given $\pbh^{n+1}$, compute $(\ubh^{n+1}, \pth^{n+1})$ with 
\algn{
\label{eq:method2-disc2} \LRa{ 2\mu \strain(\ubh^{n+1}), \strain(\vb) } - \LRa{\pth^{n+1}, \div \vb} &= \LRa{\fb^{n+1}, \vb}, \\
\label{eq:method2-disc3} - \LRa{\Div \ubh^{n+1}, \qt} - \LRa{\laminv \pth^{n+1}, \qt} &= - \LRa{\laminv \alb \cdot \pbh^{n+1}, \qt},
}
for any $\vb \in \Vb$ and $\qt \in Q_{t,h}$
 \item[{\bf Step 3}] Repeat {\bf Step 1} with $n \leftarrow n+1$
\end{itemize}

%To see the error analysis of this partitioned numerical scheme, let us consider continuous equations,
%\subeqns{eq:method2-conti}{
%\label{eq:method2-conti1} &\LRa{ 2\mu \e(\ub^{n}), \e(\vb) } + \LRa{\pt^{n}, \div \vb} = \LRa{\fb^{n}, \vb}, \\
%\label{eq:method2-conti2} &\LRa{\Div \ub^{n}, \qt} - \LRa{\laminv \pt^{n}, \qt} = \LRa{\laminv \alb \cdot \pb^{n}, \qt}, \\
%\label{eq:method2-conti3} &\LRa{ s_i \frac{\dpii^{n+1} + \dpii^n }{2} , \qii} + \LRa{\alpha_i \laminv \alb \cdot \frac{\dpb^{n+1} + \dpb^n}{2} + S_i \LRp{ \frac{\pb^{n+1} + \pb^n}{2} } , \qii } \\
%&\notag \quad + \ahi \LRp{\frac{\pii^{n+1} + \pii^{n}}{2} , \qii } \\
%&\notag \qquad = \LRa{\frac{g^n + g^{n+1}}{2},\qii} - \LRa{\alpha_i \laminv \frac{\dpt^{n+1} + \dpt^{n}}{2}, \qii}, \qquad  1 \le i \le A,
%}
%for $\vb \in \Vbh$, $\qt \in \Qth$, $\qii \in \Qih$. 

\bigskip

\begin{theorem} Let $(\ub, \pt, \pb)$ be an exact solution of \eqref{eq:mpet:vf} for compatible initial data $(\ub^0, \pt^0, \pb^0)$. For numerical initial data $(\ubh^0, \pth^0, \pbh^0)$ satisfying \eqref{eq:initial-assumption1}, suppose that $\{ (\ubh^n, \pth^n, \pbh^n) \}$ is a numerical solution obtained by {\bf Method 2}. Assuming that the exact solutions are sufficiently regular, the following hold
\algns{
\| \ub^n - \ubh^n \|_{\Vb} + \| \pt^n - \pth^n \|_{\Qt} + \| \pb^n - \pbh^n \|_S &\lesssim (\lap t)^3 + h^{k_{\ub}} + h^{k_p}, \\
\| \pb^n - \pbh^n \|_{A} &\lesssim (\lap t)^{\frac{5}{2}} + (\lap t)^{-\half} (h^{k_{\ub}} + h^{k_p} ). 
}
\end{theorem}
\begin{remark}
Compared to the elasticity-then-diffusion method, this method has one higher order convergence of time discretization errors.
\end{remark}

\begin{proof} Since the proof is similar to the one of Theorem~\ref{thm:method1}, we sketch the proof without full details. 
As before, interpolation errors denoted by $e_{\sigma}^{I,n}$ for unknown $\sigma$ are of optimal order, so it is enough to show the estimates for the errors of $e_{\sigma}^{h,n}$. 

The differences of \eqref{eq:method1-conti} and \eqref{eq:method2-disc2}, \eqref{eq:method2-disc3}, \eqref{eq:method2-disc1} with index $l$ are 
\subeqns{eq:method2-diff}{
\label{eq:method2-diff1} 
&\LRa{ 2\mu \strain(\eu^l ), \strain(\vb) } - \LRa{\ept^l , \div \vb} = 0 , \\
\label{eq:method2-diff2} 
&-\LRa{\Div \eu^l , \qt} - \LRa{\laminv \ept^{l}, \qt} = - \LRa{\laminv \alb \cdot \epb^{l} , \qt}, \\
\label{eq:method2-diff3} 
& -\LRa{ s_i \LRp{ \dpii^{l+\half} - \frac{\piih^{l+1} - \piih^l}{\lap t} } , \qii} \\
&\notag - \LRa{\alpha_i \laminv \alb \cdot \LRp{ \dpb^{l+\half} - \frac{\pbh^{l+1} - \pbh^{l} }{\lap t} } + \xi_i \LRp{ \epb^{l+\half} } , \qii } - \ahi \LRp{ \epi^{l+\half} , \qii } \\
&\notag \qquad = - \LRa{\alpha_i \laminv \LRp{ \dpt^{l+\half} - \frac{\pth^{l} - \pth^{l-1}}{\lap t} }, \qii}, \qquad  1 \le i \le N.
}
By the interpolation operators $\Pi_h^{\Vb}$, $\Pi_h^{\Qt}$, $\Pi_h^{\Qb}$, and by \eqref{eq:err-decomp}, we can obtain error equations from \eqref{eq:method2-diff} which are 
\subeqns{eq:method2-err}{
\label{eq:method2-err1} 
&\LRa{ 2\mu \strain(\eu^{h,l} ), \strain(\vb) } - \LRa{\ept^{h,l} , \div \vb} = 0 , \\
\label{eq:method2-err2} 
&-\LRa{\Div \eu^{h,l} , \qt} - \LRa{\laminv \ept^{h,l}, \qt} = - \LRa{\laminv \alb \cdot (\Ib_2^{l} + \epb^{h,l}), \qt}, \\
\label{eq:method2-err3} 
&-\LRa{ s_i \Dpi^{l} , \qii} - \LRa{\alpha_i \laminv \alb \cdot \Dpb^l  + \lap t\, \xi_i \LRp{ \epb^{h,l+\half} } , \qii } - \lap t\, \ahi \LRp{\epi^{h,l+\half} , \qii } \\
&\notag \quad = - \lap t \LRp{ \LRa{ s_i \Ib_{3,i}^l, \qii} + \LRa{ \alpha_i \laminv \alb \cdot \Ib_3^l, \qii} } \\
&\notag \qquad - \LRa{\alpha_i \laminv (\lap t I_5^l + \Dpt^{l-1}) , \qii}, \quad  1 \le i \le N
}
for $l \ge 1$ where
\algns{
%I_{1,i} &= \frac{\dpii^{n+1} + \dpii^n }{2} - \frac{\Pi \pii^{n+1} - \Pi \pii^n}{\lap t}, \\
%I_2 &= \alb \cdot \LRp{ \frac{\dpb^{n+1} + \dpb^n}{2} - \frac{\Pi \pb^{n+1} - \Pi \pb^{n} }{\lap t} }, \\
%I_{3,i} &= S_i \LRp{ \frac{\epb^{I,n+1} + \epb^{I,n}}{2} }, \\
I_5^l &=  \dpt^{l+\half} - \frac{\Pi \pt^{l} - \Pi \pt^{l-1}}{\lap t} .
}

% \subeqns{eq:method1-errD}{
% \label{eq:method1-errD1} 
% &\LRa{ 2\mu \e(\Dub^{n} ), \e(\vb) } + \LRa{\Dpt^{n} , \div \vb} = 0 , \\
% \label{eq:method1-errD2} 
% &\LRa{\Div \Dub^{n} , \qt} - \LRa{\laminv \Dpt^{n}, \qt} = \LRa{\laminv \alb \cdot (\epb^{I,n+1} - \epb^{I,n} + \Dpb^{n}), \qt}, \\
% \label{eq:method1-errD3} 
% &\LRa{ s_i \Dpi^{n} , \qii} + \LRa{\alpha_i \laminv \alb \cdot \Dpb^n  + \lap t\, S_i \LRp{ \Spb^n } , \qii } + \lap t\, \ahi \LRp{\Spi^n , \qii } \\
% &\notag \qquad = \lap t ( I_1(\qii) + I_2(\qii) + I_3(\qii) + I_4(\qii) )- \LRa{\alpha_i \laminv  \Dpt^{n-1} , \qii}, \quad  1 \le i \le A.
% }

For the difference of the equations \eqref{eq:method2-err1} and \eqref{eq:method2-err2} with indices $l+1$ and $l$, 
we take $\vb = \Dub^l + \delta \wb^l$, $\qt = -\Dpt^l$, and take $\qii = -\Dpi^l$ in \eqref{eq:method2-err3}.
%\algns{
%& \| \Dpi^{n} \|_{s_i}^2 + \LRa{\laminv \alb \cdot \Dpb^n , \alpha_i \Dpi^n} + \lap t \LRa{\xi_i \LRp{ S_{\pb}^n } , \Dpi^n } + \ahi\LRp{S_{\pii}^n , \Dpi^n } \\
%&\notag \qquad = \lap t \LRp{ \LRa{ s_i \Ib_{3,i}, \Dpi^n} + \LRa{ \alpha_i \laminv \alb \cdot \Ib_3, \Dpi^n} } - \LRa{\laminv  (I_5^n + \Dpt^{n-1}) , \alpha_i \Dpi^n }, \quad  1 \le i \le N.
%}
%Note that $\ahi \LRp{S_{\pii}^n, \Dpi^n } = \half \LRp{\| \epi^{h,n+1} \|_{\ahi}^2 - \| \epi^{h,n} \|_{\ahi}^2 }$ and 
%\algns{
%\sum_{i=1}^A \LRa{\xi_i \LRp{\Spb^n}, \Dpi^n} &= \sum_{i=1}^A \sum_{j=1}^A (\xiji (\Spi^n - S_{p_j}^n), \Dpi^n) \\
%&= \half \sum_{i=1}^A \sum_{j=1}^A \LRs {\| \epi^{h,n+1} - e_{p_j}^{h,n+1} \|_{\xiji}^2 - \| \epi^{h,n+1} - e_{p_j}^{h,n+1} \|_{\xiji}^2 } .
%}
The sum of these equations yield
\algns{
& C_0 \| \Dub^{l} \|_{\Vb}^2 + C_0 \| \Dpt^l \|_{\Qt}^2 + \| \Dpt^l \|_{\laminv}^2 + \| \Dpb^l \|_{S}^2  + \frac{\lap t}2 \| \epb^{h,l+1} \|_A^2 \\
&\quad \le  \frac{\lap t}2 \| \epi^{h,l} \|_A^2 - \LRa{\laminv \alb \cdot (\Ib_2^{l+1} - \Ib_2^{l}) , \Dpt^l } - \LRa{ \laminv \alb \cdot \Dpb^l, \Dpt^l + \Dpt^{l-1} } \\
&\qquad + {\lap t} \LRp{ S \LRp{ \Ib_{3}^l , \Dpi^l} + \LRa{ \laminv I_5^l, \alb \cdot \Dpb^l}  } .
}
We can use an argument completely analogous to the proof of \eqref{eq:l2-D-estm} and get 
\algn{
\notag &\sum_{l=1}^{n-1} \LRp{ \| \Dub^{l} \|_{\Vb}^2 + \| \Dpt^l \|_{\Qt}^2 + \| \Dpb^l \|_{S}^2 } + \lap t \| \epb^{h,n} \|_A^2  \\
\label{eq:method2-D-estm} &\quad \lesssim \sum_{l=1}^{n-1} \LRs{ \| \Ib_2^{l+1} - \Ib_2^l \|_{\laminv}^2 + (\lap t)^2 \LRp{ \| \Ib_3^l \|_S^2 + \| I_5^l \|_{\laminv}^2 } } \\
\notag &\qquad + \lap t \| \epi^{h,0} \|_A^2 +  \| \Dpt^0 \|_{\laminv}^2 \\
\notag &\quad \lesssim (\lap t)^6 + h^{2k_{\ub}} + h^{2k_p}
}
with implicit constants independent of $h$. This gives the assertion on $\| \pb^n - \pbh^n \|_A$. 

We take $\vb = \Dub^{l-1}$ for the sum of \eqref{eq:method2-err1} with indices $l+1$, $l$, and 
$\qt = - 2\ept^{h,l-\half} $ in the difference of the equation \eqref{eq:method2-err2} with indices $l+1$, $l$,
and $\qii = -2 \epi^{h,l+\half} $ in \eqref{eq:method2-err3}. If we add these equations altogether, then we get 
\algns{
& \| \eu^{h,l+1} \|_{\Vb}^2 - \| \eu^{h,l} \|_{\Vb}^2 + \| \ept^{h,l+1} - \alb \cdot \epb^{h,l+1} \|_{\laminv}^2 - \| \ept^{h,l} - \alb \cdot \epb^{h,l} \|_{\laminv}^2  \\
& + \sum_{i=1}^N \LRp{ \| \epi^{h,l+1} \|_{s_i}^2 - \| \epi^{h,l} \|_{s_i}^2 } + 2 \lap t \| \epb^{h,l+\half} \|_A^2  \\
&\quad = \LRa{\laminv \alb \cdot (\Ib_2^{l+1} - \Ib_2^{l}), \ept^{h, l + \half} } + 2 \lap t \LRp{ S\LRp{ \Ib_{3}^l , \epb^{h,l+\half}} + \LRa{ \laminv I_5^l, \alb \cdot \epb^{h,l+\half}}  } \\
&\qquad - 2 \LRa{\laminv  \LRp{ \Dpt^{l-1} - \Dpt^l }  , \alb \cdot \epb^{h,l+\half} } .
}
%With ${X}_{l}^2 := \| \eu^{h,l} \|_{\Vb}^2 + \| \ept^{h,l} - \alb \cdot \epb^{h,l} \|_{\laminv}^2 + \sum_{\| \epi^{h,n+1} \|_{D}^2$ and
With $X_l$, $Y_l$ in \eqref{eq:Xn}, \eqref{eq:Yn}, we have 
\algns{
X_{l+1}^2 - X_l^2 + Y_l^2 &= \LRa{\laminv \alb \cdot (\Ib_2^{l+1} - \Ib_2^{l}), \ept^{h, l+\half} }\\
&\quad + 2 \lap t \LRp{ S \LRp{ \Ib_{3}^l , \epb^{h,l+\half}} + \LRa{ \laminv I_5^l, \alb \cdot \epb^{h,l+\half}}  } \\
&\quad - 2 \LRa{\laminv  \LRp{ \Dpt^{l-1} - \Dpt^l } , \alb \cdot \epb^{h,l+\half} } .
}
Using 
\algns{
&2 \sum_{l=1}^{n-1} \LRa{\laminv \LRp{ \Dpt^{l-1} - \Dpt^l}, \alb \cdot \epb^{h,l+\half } }\\
& = \LRa{ \laminv \Dpt^0, \alb \cdot \epb^{h, \frac{3}{2}} } + \sum_{l=1}^{n-2} \LRa{\laminv \Dpt^l, \alb \cdot \LRp{ \Dpb^{l+1} - \Dpb^l } } - \LRa{ \laminv \Dpt^{n-1}, \alb \cdot \epb^{h, n-\half} } ,
}
by a similar argument for the estimate of $X_n$ in the proof of Theorem~\ref{thm:method1}, we can estimate 
${X}_n + \LRp{ \sum_{l=1}^{n-1} Y_l^2 }^\half$ by a linear combination of $X_1$, 
\algns{
 \sum_{l=1}^{n-1} \LRp{  \| \Ib_2^{l+1} - \Ib_2^{l} \|_{\laminv} + 8 \lap t \LRs{ \| \Ib_{3}^l \|_{S} + \| I_5^l \|_{\laminv} } } , \\
 \| \Dpt^0 \|_{\laminv} + \LRp{ \sum_{l=1}^{n-1} \LRs{ \| \Dpt^l \|_{\laminv}^2  + \| \alb \cdot \Dpb^{l} \|_{\laminv}^2 } }^\half.
}
These terms are bounded by $(\lap t)^3 + h^{k_{\ub}} + h^{k_p}$ by the estimates of interpolation error terms, \eqref{eq:method2-D-estm}, and Theorem~\ref{thm:init-estm}. As in the proof of Theorem~\ref{thm:method1} we can prove the asserted error estimates from the estimate of $X_n$. 
\end{proof}

%\input{./results/two-net-partitioned_dte1.tex}
%\input{./results/two-net-partitioned_etd1.tex}

%% method = method=dte
%% deg = 1
\begin{table}[ht]
	\begin{center}
		\begin{tabular}{c | c c | c c | c c | c c} 
			\multirow{3}{*}{$M$} & \multicolumn{2}{c}{$\| \ub - \ubh \|_{1}$}& \multicolumn{2}{c}{$\| \pt - \pth \|_0$}& \multicolumn{2}{c}{$\| p_1 - p_{1,h} \|_1$}& \multicolumn{2}{c}{$\| p_2 - p_{2,h} \|_1$}\\ 
\cline{2-9}& error & rate & error & rate & error & rate & error & rate \\
\hline
8 & 1.290e+0  & -  & 2.146e-1  & -  & 2.661e-1  & -  & 5.323e-1  & - \\
16 & 3.195e-1  & 2.01  & 3.898e-2  & 2.46  & 1.865e-1  & 0.51  & 3.729e-1  & 0.51 \\
32 & 7.700e-2  & 2.05  & 8.856e-3  & 2.14  & 1.059e-1  & 0.82  & 2.118e-1  & 0.82 \\
64 & 1.872e-2  & 2.04  & 2.154e-3  & 2.04  & 5.599e-2  & 0.92  & 1.120e-1  & 0.92 \\
128 & 4.603e-3  & 2.02  & 5.333e-4  & 2.01  & 2.873e-2  & 0.96  & 5.747e-2  & 0.96 \\
		\hline
		\end{tabular} 
		\caption{The errors and convergence rates given by the diffusion-then-elasticity scheme with $k=1$} 
		\label{table:dte1} 
	\end{center} 
\end{table} 

%% method = method=etd
%% deg = 1
\begin{table}[th]
	\begin{center}
		\begin{tabular}{c | c c | c c | c c | c c} 
			\multirow{3}{*}{$M$} & \multicolumn{2}{c}{$\| \ub - \ubh \|_{1}$}& \multicolumn{2}{c}{$\| \pt - \pth \|_0$}& \multicolumn{2}{c}{$\| p_1 - p_{1,h} \|_1$}& \multicolumn{2}{c}{$\| p_2 - p_{2,h} \|_1$}\\ 
\cline{2-9}& error & rate & error & rate & error & rate & error & rate \\
\hline
8 & 1.290e+0  & -  & 2.146e-1  & -  & 2.661e-1  & -  & 5.323e-1  & - \\
16 & 3.195e-1  & 2.01  & 3.898e-2  & 2.46  & 1.865e-1  & 0.51  & 3.729e-1  & 0.51 \\
32 & 7.700e-2  & 2.05  & 8.856e-3  & 2.14  & 1.059e-1  & 0.82  & 2.118e-1  & 0.82 \\
64 & 1.872e-2  & 2.04  & 2.154e-3  & 2.04  & 5.599e-2  & 0.92  & 1.120e-1  & 0.92 \\
128 & 4.603e-3  & 2.02  & 5.333e-4  & 2.01  & 2.873e-2  & 0.96  & 5.747e-2  & 0.96 \\
		\hline
		\end{tabular} 
		\caption{The errors and convergence rates given by the elsticity-then-diffusion scheme with $k=1$} 
		\label{table:etd1} 
	\end{center} 
\end{table}

%-------------------------------------------------------------------------------
%-------------------------------------------------------------------------------
\section{Numerical convergence experiments}
\label{sec:numerics:convergence}

In this section, we present a set of numerical results to illustrate
the performances of the new partitioned methods. For this we construct a manufactured smooth solution 
and compute convergence of errors of the numerical solutions.
All numerical simulations in this section were run using the FEniCS
finite element software~\cite{AlnaesEtAl2015} (version 2017.2.0). 
For simplicity we only consider the two-network case, i.e., $N=2$ in our numerical experiments.
The exact solutions in our experiments are
\algns{
\ub &= \pmat{\sin(2 \pi y) (-1 + \cos(2 \pi x)) \sin t + (\mu + \lambda)^{-1} \sin(\pi x) \sin(\pi y)) \sin t \\
         \sin(2 \pi x) (1 - \cos(2 \pi y)) \sin t + (\mu + \lambda)^{-1} \sin(\pi x) \sin(\pi y)) \sin t }, \\
p_1 &= -\sin(\pi x) \sin(\pi y) \cos t, \\
p_2 &= -2 \sin(\pi x) \sin(\pi y) \cos t .
}
The physical parameters are given as $E = 1.0$, $\nu = 0.49999$, $\alpha_1 = \alpha_2 = 1$, $s_1 = s_2 = 1$, $K_1 = K_2 = 1$, 
and corresponding $\mu$ and $\lambda$ are computed by $\mu = E/(2(1+\nu))$ and $\lambda = \nu E/((1 - 2\nu)(1+\nu))$.
The domain $\Omega$ is $[0,1] \times [0,1]$. To check convergence of errors we use nested structured triangular meshes obtained by dividing $\Omega$ into $M \times M$ rectangles ($M = 8, 16, 32, 64, 128$), i.e., the relation $h \sim 1/M$ holds with uniform implicit constants.  
For simplicity we impose Dirichlet boundary conditions of $\ub$ on the vertical sides of $\Omega$ whereas we impose Dirichlet boundary conditions of $p_1$ and $p_2$ on the whole boundary of $\Omega$. The time step size $\lap t$ is taken as $\lap t = 1/M$ in all experiments, so $\lap t \sim h$ holds. 
In all the numerical experiments we use the Taylor--Hood elements with polynomial degrees of $k+1$ and $k$ for $\Vbh$ and $\Qth$, and the Lagrange finite elements of degree $k$ for $Q_{1,h}$ and $Q_{2,h}$. 

In Table~\ref{table:dte1} and Table~\ref{table:etd1}, we present the errors of variables computed at $t=1$ with $k=1$ for the diffusion-then-elasticity (DTE) and the elasticity-then-diffusion (ETD) schemes, respectively. The convergence rates of the errors of $p_1$ and $p_2$ in the $H^1$ norm are bounded by 1 because of the best approximation property of linear polynomials. Other errors show convergence rates higher than the rates expected by our error analysis.

%\input{./results/two-net-partitioned_dte2.tex}
%\input{./results/two-net-partitioned_etd2.tex}

%% method = method=dte
%% deg = 2
\begin{table}[th]
	\begin{center}
		\begin{tabular}{c | c c | c c | c c | c c} 
			\multirow{3}{*}{$M$} & \multicolumn{2}{c}{$\| \ub - \ubh \|_{1}$}& \multicolumn{2}{c}{$\| \pt - \pth \|_0$}& \multicolumn{2}{c}{$\| p_1 - p_{1,h} \|_1$}& \multicolumn{2}{c}{$\| p_2 - p_{2,h} \|_1$}\\
\cline{2-9}& error & rate & error & rate & error & rate & error & rate \\
\hline
8 & 2.682e-1  & -  & 3.405e-2  & -         & 4.082e-2  & -     & 8.165e-2  & - \\
16 & 3.153e-2  & 3.09  & 3.615e-3  & 3.24  & 1.440e-2  & 1.50  & 2.880e-2  & 1.50 \\
32 & 3.698e-3  & 3.09  & 4.082e-4  & 3.15  & 4.098e-3  & 1.81  & 8.196e-3  & 1.81 \\
64 & 4.451e-4  & 3.05  & 4.865e-5  & 3.07  & 1.084e-3  & 1.92  & 2.168e-3  & 1.92 \\
128 & 5.454e-5  & 3.03  & 5.943e-6  & 3.03  & 2.781e-4  & 1.96  & 5.563e-4  & 1.96 \\
		\hline
		\end{tabular} 
		\caption{The errors and convergence rates given by the diffusion-then-elasticity scheme with $k=2$} 
		\label{table:dte2} 
	\end{center} 
\end{table} 

%% method = method=etd
%% deg = 2
\begin{table}[th]
	\begin{center}
		\begin{tabular}{c | c c | c c | c c | c c} 
			\multirow{3}{*}{$M$} & \multicolumn{2}{c}{$\| \ub - \ubh \|_{1}$}& \multicolumn{2}{c}{$\| \pt - \pth \|_0$}& \multicolumn{2}{c}{$\| p_1 - p_{1,h} \|_1$}& \multicolumn{2}{c}{$\| p_2 - p_{2,h} \|_1$}\\
\cline{2-9}& error & rate & error & rate & error & rate & error & rate \\
\hline
8 & 2.682e-1   & -     & 3.405e-2  & -     & 4.082e-2  & -     & 8.165e-2  & - \\
16 & 3.153e-2  & 3.09  & 3.615e-3  & 3.24  & 1.440e-2  & 1.50  & 2.880e-2  & 1.50 \\
32 & 3.698e-3  & 3.09  & 4.082e-4  & 3.15  & 4.098e-3  & 1.81  & 8.196e-3  & 1.81 \\
64 & 4.451e-4  & 3.05  & 4.865e-5  & 3.07  & 1.084e-3  & 1.92  & 2.168e-3  & 1.92 \\
128 & 5.454e-5  & 3.03  & 5.943e-6  & 3.03  & 2.781e-4  & 1.96  & 5.563e-4  & 1.96 \\
		\hline
		\end{tabular} 
		\caption{The errors and convergence rates given by the elsticity-then-diffusion scheme with $k=2$} 
		\label{table:etd2} 
	\end{center} 
\end{table} 

To confirm the second order convergence in time, we test the numerical methods with $k=2$, and the results for DTE and ETD schemes are given in Table~\ref{table:dte2} and Table~\ref{table:etd2}. In the results we observe that some errors show the convergence rates which are higher than the rates expected by the error analysis. The expected convergence rates of $\| p_1 - p_{1,h} \|_1$ and $\| p_2 - p_{2,h} \|_1$ are $\frac 32$ in the error analysis but the numerical results show second order convergence. We conjecture that there is an improved way to analyze the errors with second order convergence of these errors but we leave it as a future research work. The convergence rates of the errors $\| \ub - \ubh \|_1$ and $\| \pt - \pth \|_0$ are also not covered by the present error analysis. Note that these errors have $\frac 12$ higher order convergence rates in the current error analysis and note also that some time error terms are of orders higher than 2 (see \eqref{eq:I3-estm} and \eqref{eq:I4-estm}). Thus we believe that the improved error analysis in our conjecture can be used to obtain these higher convergence rates of $\| \ub - \ubh \|_1$ and $\| \pt - \pth \|_0$. 

%The previous numerical results show that second order may not be the optimal convergence rate of time discretization errors, so 
Finally, we run experiments with $k=3$ and the results are presented in Table~\ref{table:dte3} and Table~\ref{table:etd3}. Interestingly, the convergence rates of $\| p_1 - p_{1,h} \|_1$ and $\| p_2 - p_{2,h} \|_1$ are approximately $\frac 52$ but those of $\| \ub - \ubh \|_1$ and $\| \pt - \pth \|_0$ are asymptotically 4. The analysis in Subsection~\ref{subsec:dte} explains the $\frac 52$ convergence of $\| p_1 - p_{1,h} \|_1$ and $\| p_2 - p_{2,h} \|_1$ in the DTE scheme but it still cannot explain the superconvergence of $\| \ub - \ubh \|_1$ and $\| \pt - \pth \|_0$ of order 4. In the ETD scheme, the superconvergent errors are observed and they are beyond the scope of the analysis in Subsection~\ref{subsec:etd}. 
%These higher order convergence rates are unexpected because it is usually believed that the monolithic time step with the Crank-Nicolson method have time discretization errors at most of second order.
At present we have no theory to explain this superconvergence, so it will be investigated in our future research.

%\input{./results/two-net-partitioned_dte3.tex}
%\input{./results/two-net-partitioned_etd3.tex}

%% method = method=dte
%% deg = 3
\begin{table}[h]
	\begin{center}
		\begin{tabular}{c | c c | c c | c c | c c} 
			\multirow{3}{*}{$M$} & \multicolumn{2}{c}{$\| \ub - \ubh \|_{1}$}& \multicolumn{2}{c}{$\| \pt - \pth \|_0$}& \multicolumn{2}{c}{$\| p_1 - p_{1,h} \|_1$}& \multicolumn{2}{c}{$\| p_2 - p_{2,h} \|_1$}\\ 
\cline{2-9}& error & rate & error & rate & error & rate & error & rate \\
\hline
8 & 4.942e-2  & -      & 8.388e-3  & -     & 4.240e-3  & -     & 8.479e-3  & - \\
16 & 3.108e-3  & 3.99  & 4.581e-4  & 4.19  & 7.292e-4  & 2.54  & 1.458e-3  & 2.54 \\
32 & 1.888e-4  & 4.04  & 2.626e-5  & 4.12  & 1.058e-4  & 2.78  & 2.114e-4  & 2.79 \\
64 & 1.150e-5  & 4.04  & 1.559e-6  & 4.07  & 1.556e-5  & 2.77  & 3.092e-5  & 2.77 \\
128 & 7.069e-7  & 4.02  & 9.467e-8  & 4.04  & 2.719e-6  & 2.52  & 5.280e-6  & 2.55 \\
		\hline
		\end{tabular} 
		\caption{The errors and convergence rates given by the diffusion-then-elasticity scheme with $k=3$} 
		\label{table:dte3} 
	\end{center} 
\end{table} 

%% method = method=etd
%% deg = 3
\begin{table}[h]
	\begin{center}
		\begin{tabular}{c | c c | c c | c c | c c} 
			\multirow{3}{*}{$M$} & \multicolumn{2}{c}{$\| \ub - \ubh \|_{1}$}& \multicolumn{2}{c}{$\| \pt - \pth \|_0$}& \multicolumn{2}{c}{$\| p_1 - p_{1,h} \|_1$}& \multicolumn{2}{c}{$\| p_2 - p_{2,h} \|_1$}\\ 
\cline{2-9}& error & rate & error & rate & error & rate & error & rate \\
\hline
8 & 4.942e-2  & -      & 8.388e-3  & -     & 4.240e-3  & -     & 8.479e-3  & - \\
16 & 3.108e-3  & 3.99  & 4.581e-4  & 4.19  & 7.292e-4  & 2.54  & 1.458e-3  & 2.54 \\
32 & 1.888e-4  & 4.04  & 2.626e-5  & 4.12  & 1.058e-4  & 2.78  & 2.114e-4  & 2.79 \\
64 & 1.150e-5  & 4.04  & 1.559e-6  & 4.07  & 1.556e-5  & 2.77  & 3.092e-5  & 2.77 \\
128 & 7.069e-7  & 4.02  & 9.467e-8  & 4.04  & 2.719e-6  & 2.52  & 5.280e-6  & 2.55 \\
		\hline
		\end{tabular} 
		\caption{The errors and convergence rates given by the elsticity-then-diffusion scheme with $k=3$} 
		\label{table:etd3} 
	\end{center} 
\end{table}

%-------------------------------------------------------------------------------
%-------------------------------------------------------------------------------
\section{Conclusions}
\label{sec:conclusion}

In this paper, we have presented two partitioned time discretization schemes
for the total-pressure-based formulation of quasi-static multiple-network poroelasticity.
We proved that the partitioned schemes are unconditionally stable and have second and third order 
convergence in time by a novel error analysis. The analyses also show that the numerical 
schemes are still robust in the limit of incompressibility and other parameter variations such as small storage coefficients. 
We also presented a number of numerical experiments to illustrate the validaty of our theoretical analysis but the numerical results often show superconvergence results in time discretization errors which are not completely understood by the current error analysis.
Thus, an improved error analysis for the schemes will be investigated in the future work.

%Our formulation introduces a single additional scalar field
%unknown, the total pressure. We prove, via energy and semi-discrete
%\emph{a priori} error estimates, that this formulation is robust in
%the limits of incompressibility ($\lambda \rightarrow \infty$) and
%vanishing storage coefficients ($c_j \rightarrow 0$), in contrast to
%standard formulations. Finally, numerical experiments support the
%theoretical results. For the numerical experiments presented here, we
%have used direct linear solvers. In future work, we will address
%iterative solvers and preconditioning of the MPET equations.

%-------------------------------------------------------------------------------
%-------------------------------------------------------------------------------

%\algns{
%&\half (f'(h) + f'(0)) - \frac{f(0) - f(-h)}{h} \\
%&= f'(0) + \half h f''(0) + \half \frac{1}{2!} h^2 f'''(0) + \cdots \\
%&\quad - \frac{1}{h} \LRp{ h f'(0) - \frac{h^2}{2} f''(0) + \frac{h^3}{3!} f'''(0) + \cdots }
%}

%\FloatBarrier

\bibliographystyle{amsplain}
%\bibliography{references}

\end{document}